\documentclass[11pt,reqno]{amsart}
\usepackage[T1]{fontenc}
\usepackage[utf8]{inputenc}
\usepackage{lmodern,amsmath,amssymb,amsthm,mathtools,mathrsfs}
\usepackage[margin=1.08in]{geometry}
\usepackage{microtype,enumitem,graphicx,tikz}
\usepackage[hidelinks,bookmarksnumbered]{hyperref}
\allowdisplaybreaks[2]
\numberwithin{equation}{section}
\newtheorem{theorem}{Theorem}[section]
\newtheorem{proposition}[theorem]{Proposition}
\newtheorem{lemma}[theorem]{Lemma}
\newtheorem{corollary}[theorem]{Corollary}
\theoremstyle{definition}

\newtheorem{problem}[theorem]{Problem}
\newtheorem{example}[theorem]{Example}
\theoremstyle{remark}
\newtheorem{remark}[theorem]{Remark}
\newcommand{\C}{\mathbb C}
\newcommand{\R}{\mathbb R}

\newcommand{\Sh}{\mathcal K}
\newcommand{\AllSh}{\mathcal S}
\newcommand{\FixSh}{\mathcal K^{\mathrm{fix}}}
\newcommand{\Inc}{\mathcal I}
\newcommand{\dd}{\,\mathrm d}
\newcommand{\Ree}{\operatorname{Re}}
\newcommand{\Imm}{\operatorname{Im}}
\newcommand{\supp}{\operatorname{supp}}
\newcommand{\conv}{\operatorname{conv}}
\newcommand{\dist}{\operatorname{dist}}
\newcommand{\divz}{\operatorname{div}_0}
\newcommand{\Area}{\operatorname{Area}}
\newcommand{\Haus}{\mathcal H}
\newcommand{\ind}{\mathbf 1}
\newcommand{\Rexp}{\R_{\exp}}
\newcommand{\Blev}{\mathcal B}
\newcommand{\rootmu}{\nu}
\newcommand{\prz}{\operatorname{pr}_z}
\DeclareMathOperator{\wind}{wind}
\hypersetup{pdftitle={The full shadow of a polynomial II. Vanishing root masses and holes},
 pdfauthor={Christian Hagg and Boris Shapiro}}
\title[The full shadow of a polynomial II]{The full shadow of a polynomial II. Vanishing root masses and holes}
\author{Christian H\"agg}
\address{Department of Mathematics, Stockholm University, SE--106 91 Stockholm, Sweden}
\email{hagg@math.su.se}
\author{Boris Shapiro}
\address{Department of Mathematics, Stockholm University, SE--106 91 Stockholm, Sweden}
\email{shapiro@math.su.se}
\date{September 15, 2026}
\subjclass[2020]{Primary 30C15, 31A35; Secondary 03C64, 26C10}
\keywords{Polynomial derivatives, root clusters, vanishing masses, zero-set convergence, holes, finite topology}
\begin{document}
\begin{abstract}
For a complex polynomial $P$, its full shadow is the closure of the set
of zeros of all derivatives of all powers $P^n$, and its limiting core
is the Hausdorff limit, as $n\to\infty$, of the set of zeros of all
derivatives of $P^n$. These sets retain geometric information that weak
limits of root-counting measures discard. We determine complete
zero-set limits for arrays of polynomials whose roots approach finitely
many locations. A cluster containing $o(D)$ roots of a degree-$D$
polynomial can survive outside the limiting support of the derivative
zeros, which depends only on the positive limiting masses: under
repeated differentiation, the zeros it produces follow a straight path,
selected by an explicit min--max condition, and preserve the cluster's
exact multiplicity. When such a location instead carries a
small positive mass $\varepsilon$, the support of the limiting zero
distribution at a fixed ratio of derivative order to degree develops an
arc on the $\sqrt{\varepsilon}$ scale with a universal local semicircle
profile. We use these mechanisms to prove that polynomial
shadows can have robust and arbitrarily complicated topology. For every
$k\ge2$, there are polynomials with $k+2$ distinct root locations whose
limiting cores and full shadows have at least $\binom{k}{2}$ holes; this
persists under coefficient perturbations and among square-free
polynomials without a reflection axis. In a four-location family, the
core and full-shadow areas tend to zero while one hole has area tending
to $8/1221$.
\end{abstract}
\maketitle
\enlargethispage{2pt}
\setcounter{tocdepth}{1}
\begingroup\small\tableofcontents\endgroup

\section{Introduction}\label{sec:intro}
Repeated differentiation is often studied through normalized
root-counting measures. This is the natural scale for zeros carrying a
positive fraction of the degree, but it erases every cluster containing
$o(D)$ roots of a degree-$D$ polynomial. For questions about the set of
all zeros of all derivatives, that loss is decisive: a cluster of
vanishing normalized mass can travel a macroscopic distance under
differentiation, and several such trajectories can change the topology
of the set swept out by all derivative orders.

Polynomial shadows retain precisely this set-level information. Starting
from a fixed polynomial $P$, one considers every derivative of every
power $P^n$. The \emph{full shadow} of $P$ is the closure of the set of
all zeros of all these derivatives, and its \emph{limiting core} is the
Hausdorff limit, as $n\to\infty$, of the set of zeros of all derivatives
of $P^n$; both are defined in Section~\ref{sec:intro-objects}. The
companion paper \cite{HS1} showed that the full shadow is the disjoint
union of the core and a countable set of isolated \emph{transient
zeros}, each occurring at only finitely many power levels $n$, and that both
sets depend continuously on the coefficients of $P$
(Proposition~\ref{prop:shadow-foundations} below). The core is
connected, and \cite{HS1} leaves open whether it must be contractible.
The present paper develops a complete zero-set theory for \emph{arrays},
that is, sequences of polynomials $H_j$ of degrees $D_j\to\infty$, whose
roots approach finitely many locations. We use it to construct
polynomials whose cores and full shadows have arbitrarily many
\emph{holes} (bounded complementary components) that persist under
coefficient perturbations.

The central point is that, for an array, the weak limit of the
normalized root-counting measures does not determine the limit of the
zero sets of the derivatives; one must also retain the limit of the sets
of root locations. Locations of positive limiting mass produce the
limiting zero distributions of the derivatives and their supports. A
location with zero limiting mass may still launch a cluster of zeros of
the derivatives; we call the zeros of such a cluster its
\emph{descendant zeros}. This cluster
travels outside those supports along a straight transport path, on a
part of the path selected by an explicit condition, and it contains
exactly as many zeros as the original cluster. This mechanism is
invisible at the level of probability measures and is the source of the
topological phenomena below; see Figure~\ref{fig:trails}. (In
\cite{BHS,HS1} the derivatives $(P^n)^{(m)}$ themselves are called
descendants of $P$.)

\subsection{Zero sets, cores, and full shadows}\label{sec:intro-objects}
For a nonzero polynomial $H$, let $Z(H)$ be its set of distinct zeros
and let
\[
 \divz H=\sum_{a\in Z(H)}\operatorname{ord}_a(H)\delta_a
\]
be its zero divisor, viewed as a measure. Here $\operatorname{ord}_a(H)$
is the multiplicity of the zero and $\delta_a$ is unit mass at $a$.
For $D=\deg H\ge1$, the probability measure $D^{-1}\divz H$ is the
\emph{normalized root measure} of $H$, and we put
\begin{equation}\label{eq:finite-order-objects}
 \mathcal Z(H)=\bigcup_{m=0}^{D-1}Z(H^{(m)}),\qquad
 \pi_H=\frac{2}{D(D+1)}\sum_{m=0}^{D-1}\divz H^{(m)}.
\end{equation}
The measure $\pi_H$ is a probability measure, since the derivative
degrees sum to $D(D+1)/2$.

The distinction between measure and set limits appears already in the
elementary family $H_D(z)=z^{D-1}(z-1)$. Its normalized root measures
tend weakly to $\delta_0$, while
\[
 \mathcal Z(H_D)=\{0,1/D,2/D,\ldots,1\}\longrightarrow[0,1]
\]
in Hausdorff distance. Thus a single root of normalized mass $1/D$ can
leave a macroscopic trace through its descendant zeros. The general
theory below identifies this trace for arbitrary finite root
configurations.

For a nonconstant polynomial $P$, its \emph{full shadow} is
\begin{equation}\label{eq:full-shadow-definition}
 \AllSh_P=\overline{\bigcup_{n\ge1}\mathcal Z(P^n)}.
\end{equation}
The power $P^n$ is an ordinary product, not a functional iterate. We
also write $\mathcal Z_{P,n}=\mathcal Z(P^n)$. Hausdorff convergence
refers to the metric
\[
 d_{\mathrm H}(E,F)=\max\left\{
 \sup_{x\in E}\dist(x,F),\ \sup_{y\in F}\dist(y,E)\right\}
\]
on nonempty compact sets. We call Hausdorff convergence of such zero
sets \emph{complete zero-set convergence}; unlike weak convergence of
normalized zero divisors, it retains every zero location.

The companion paper proves the following structural facts. We state them
in the notation used here so that the present paper can be read without
assuming prior familiarity with the core--transient decomposition.
Convergence $P_j\to P$ of monic polynomials of a fixed degree means
convergence of their coefficients.

\begin{proposition}[Core convergence, stability, and transient zeros]
\label{prop:shadow-foundations}
\cite[Theorem~5.3, Corollary~5.4, Theorems~5.6--5.7, and
Corollaries~5.8 and~5.12]{HS1}
Let $P$ be a nonconstant polynomial. There is a nonempty compact set
$\Sh_P\subseteq\conv Z(P)$ such that
\begin{equation}\label{eq:core-definition}
 \mathcal Z_{P,n}\longrightarrow\Sh_P\qquad(n\to\infty).
\end{equation}
The set $\Sh_P$ is connected. If $P$ has at least two distinct zeros,
then $\Sh_P$ is perfect and
\[
 \operatorname{acc}\AllSh_P=\Sh_P=\operatorname{acc}\Sh_P,
\]
where $\operatorname{acc}E$ denotes the set of limits of sequences of
pairwise distinct points of $E$.

Let
\[
 T_P:=\AllSh_P\setminus\Sh_P.
\]
Then
\begin{equation}\label{eq:core-transients}
 \AllSh_P=\Sh_P\,\dot\cup\,T_P,
 \qquad
 T_P\subseteq\bigcup_{n\ge1}\mathcal Z_{P,n}.
\end{equation}
Every point of $T_P$ is isolated in $\AllSh_P$ and occurs at only
finitely many power levels $n$. The set $T_P$ is countable and locally
finite in $\C\setminus\Sh_P$. Consequently,
\begin{equation}\label{eq:shadow-interior-area}
 \operatorname{int}\AllSh_P=\operatorname{int}\Sh_P,
 \qquad
 \Area(\AllSh_P)=\Area(\Sh_P),
 \qquad
 \partial\AllSh_P=\partial\Sh_P\cup T_P.
\end{equation}
For every component $\Omega$ of $\C\setminus\Sh_P$, the set
$\Omega\setminus T_P$ is a component of
$\C\setminus\AllSh_P$. This correspondence is a bijection of
complementary components; it preserves boundedness and area, and
$\overline{\Omega\setminus T_P}=\overline\Omega$.
Thus the core and the full shadow have exactly the same holes.

The construction is stable in each fixed degree. If $P_j\to P$ are
monic of degree $d\ge2$, then
\begin{equation}\label{eq:varying-all-orders}
 d_{\mathrm H}(\mathcal Z_{P_j,j},\Sh_P)\longrightarrow0.
\end{equation}
The maps $P\mapsto\Sh_P$ and $P\mapsto\AllSh_P$ are
Hausdorff-continuous, including at root collisions. More precisely, for
every compact family $\mathscr P$ of monic degree-$d$ polynomials,
\begin{equation}\label{eq:uniform-all-orders}
 \sup_{P\in\mathscr P}
 d_{\mathrm H}(\mathcal Z_{P,n},\Sh_P)\longrightarrow0.
\end{equation}
\end{proposition}

We call $\Sh_P$ the \emph{limiting core}, or simply the \emph{core}, of
$P$, the points of $T_P$ its \emph{transient zeros}, and $\Sh_P$ and
$\AllSh_P$ the two \emph{shadows} of $P$. The proposition shows that
transient zeros record genuine finite-level exceptions but do not affect
the interior, area, or holes of the full shadow.

Multiplication by a nonzero scalar changes none of our objects, so
polynomials may be taken monic. Write
\begin{equation}\label{eq:root-data}
 P(z)=\prod_{j=1}^s(z-a_j)^{r_j},\qquad
 d=\sum_{j=1}^s r_j,\qquad p_j=\frac{r_j}{d},
\end{equation}
where the $a_j$ are distinct. The normalized root measure and the
centroid of $P$ are
\[
 \rootmu_P=\frac1d\divz P=\sum_{j=1}^s p_j\delta_{a_j},
 \qquad \mathfrak c_P=\sum_{j=1}^s p_j a_j.
\]
Throughout the paper, $t$ is the derivative order divided by the
\emph{total degree}: $t=m/D$ for a degree-$D$ polynomial and
$t=m/(dn)$ for $P^n$. Thus the unnormalized ratio $m/n\in(0,d)$ used in
\cite{HS1} equals $dt$. Weak convergence of measures is written
$\Longrightarrow$. We use $\dd A$, $\Area$, and $\Haus^1$ for planar
Lebesgue measure, planar area, and one-dimensional Hausdorff measure,
respectively.

Consider an array $H_j$ of degrees $D_j\to\infty$ whose normalized root
measures converge weakly to a finitely supported probability measure
$\eta=\sum_{a\in A}p_a\delta_a$, all $p_a>0$, while the sets $Z(H_j)$
of distinct roots converge in Hausdorff distance to a finite set $B$.
Necessarily $A\subseteq B$, and the inclusion may be strict. The points
of $B\setminus A$ are the \emph{vanishing-mass locations}: each of them
attracts roots of $H_j$, but only a vanishing fraction of them. A main
conclusion of this paper is that complete zero-set limits of the
derivatives are governed by the pair $(\eta,B)$, not by $\eta$ alone.
If $m_j/D_j\to t\in(0,1)$, the normalized root measures of
$H_j^{(m_j)}$ converge to a probability measure $\mu_{t,\eta}$ that
depends on $\eta$ alone; its support $S_{t,\eta}$ is the
\emph{fixed-ratio support} (Proposition~\ref{prop:weighted-foundations}
and Theorem~\ref{thm:finite-cluster-universality}). For the powers of a
fixed polynomial $P$ these are the supports $S_{t,\rootmu_P}$, and the
core $\Sh_P$ is the closure of their union over $0<t<1$; as $t\to1$ they
shrink to the centroid $\mathfrak c_P$. For general $\eta$ we call this
closure $S_\eta$ the \emph{limiting core} of $\eta$.

\subsection{The geometric conclusions}
A \emph{hole} of a compact planar set is a bounded connected component
of its complement. A set is \emph{contractible} if its identity map can
be continuously deformed within the set to a constant map. The
foundational result above gives connectedness of the limiting core, but
it leaves open whether the core must be contractible. Our examples show
that it need not be: the holes occur in the connected core itself and do
not arise from isolated transient zeros.

We regard monic degree-$d$ polynomials as the coefficient space
$\C^d$, with its usual topology. A polynomial is \emph{square-free}
if all its zeros are simple. A \emph{reflection axis} is a line whose
reflection preserves the root multiset, including multiplicities.
Excluding multiple roots and reflection axes shows that the holes below
are not artifacts of multiplicity or symmetry.

\begin{theorem}[Robust holes]\label{thm:intro-holes}
For each $k\ge2$ and every sufficiently large integer $N$, the polynomial
\[
 H_{N,k}(z)=(z^2-1)^N\prod_{j=1}^k(z-j-2-i)
\]
has at least $\binom{k}{2}$ holes in both $\Sh_{H_{N,k}}$ and
$\AllSh_{H_{N,k}}$. This lower bound persists on an open coefficient
neighborhood, which contains square-free polynomials without a
reflection axis.

The four-location family
\[
 P_N(z)=(z^2-1)^N(z-2-i)(z-3-i)
\]
has a noncontractible core for every sufficiently large $N$. For either
shadow, the complementary component containing
$q_\diamond=2/3+i/30$ has area tending to $8/1221$, whereas
\[
 \Area(\Sh_{P_N})=\Area(\AllSh_{P_N})\longrightarrow0.
\]
\end{theorem}

The number of root locations must grow with the number of holes: by
Theorem~\ref{thm:atomic-complexity}, the number of holes of the core of
$P$ is bounded in terms of the number of distinct roots of $P$.
Figure~\ref{fig:trails} shows the zeros of all derivatives of $P_N$ for
$N=200$, and Figure~\ref{fig:hole} shows the core of $P_N$ for
$N=1000$. The four-location example exhibits a particularly sharp
contrast: the areas of the core and the full shadow tend to zero, while
a complementary component retains a positive limiting area.

\begin{figure}[tp]
\centering
\includegraphics[width=\textwidth]{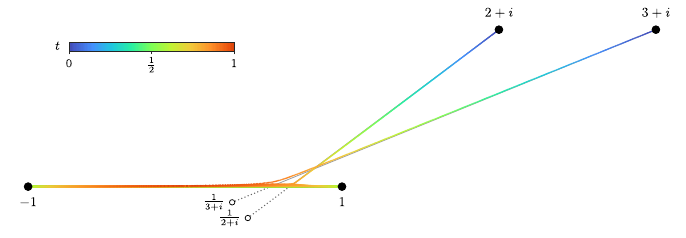}\\[1pt]
{\small (a)}\\[5pt]
\includegraphics[width=\textwidth]{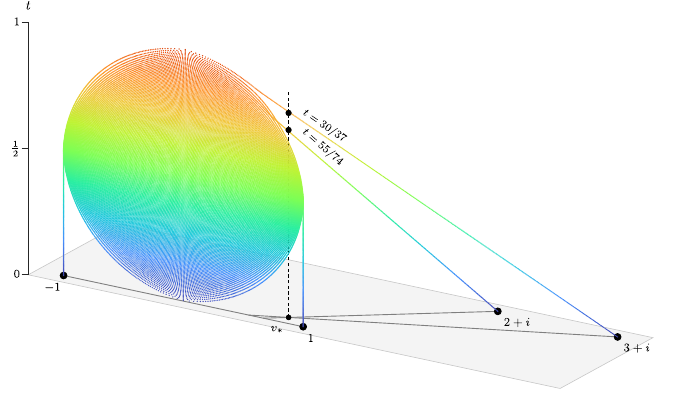}\\[1pt]
{\small (b)}
\caption{All zeros of all derivatives $P_N^{(m)}$, $0\le m<D=2N+2$, of
the polynomial $P_N$ of Theorem~\ref{thm:intro-holes} with $N=200$,
colored by $t=m/D$. (a) Each simple root $b=2+i,3+i$ sends one zero per
order along its transport segment $\theta_b(t)=(1-t)b+t/b$ (gray) until
it reaches $[-1,1]$; dotted: the continuations up to
$\theta_b(1)=1/b$, which are not selected. (b) The same zeros lifted to
height $t$. Over $[-1,1]$ they fill the ellipse
$x^2+4(t-\frac12)^2\le1$ swept out by the fixed-ratio supports
$[-2\sqrt{t(1-t)},2\sqrt{t(1-t)}]$, with the atoms $\pm1$ for
$t<\frac12$. The trails pass over the crossing $v_*=(30+4i)/37$ of
their projections (gray) at the different ratios $\frac{55}{74}$ and
$\frac{30}{37}$.}
\label{fig:trails}
\end{figure}

The geometric theorem is a consequence of a more general complete-limit
result for arrays. In the setting above, let
\[
 C_\eta(z)=\int\frac{\dd\eta(\xi)}{z-\xi}
\]
be the Cauchy transform of $\eta$, and let $m_j/D_j\to t\in(0,1)$.
Every limit of zeros of $H_j^{(m_j)}$ outside $S_{t,\eta}$, that is,
every limit of \emph{exterior} descendant zeros, is the point
$\theta_b(t)$, at the same ratio $t$, of a straight \emph{transport path}
\[
 \theta_b(t)=b-\frac{t}{C_\eta(b)},
\]
where $b\in B\setminus A$ and $C_\eta(b)\ne0$. Heuristically, Cauchy's
integral for $H_j^{(m_j)}(z)/m_j!$ has the limiting phase
$w\mapsto\sum_ap_a\log(w-a)-t\log(w-z)$, and $\theta_b(t)$ is the
observation point $z$ at which $b$ is a saddle point of this phase. The
real part of the phase is a landscape that tends to $+\infty$ at $z$ and
at infinity. Its \emph{barrier} is the min--max height of curves around
$z$, and the saddles at which the regions above the barrier around $z$
and around infinity meet are called \emph{connecting}, or
\emph{selected} (Figure~\ref{fig:selection}). When $b$ is the selected
saddle, the roots clustered near $b$ sit at the dominant saddle point of
the integral, and they reappear, one for one, as zeros near
$\theta_b(t)$. This barrier-selection condition determines which parts
of the transport paths actually occur; for $P_N$ they are the two
straight trails in Figure~\ref{fig:trails}(a). At every selected
exterior point, the descendant cluster contains exactly as many zeros as
the corresponding input cluster. The cluster size may tend to infinity,
provided it is $o(D_j)$; no separation, equal-multiplicity, or
simplicity assumption inside the cluster is required. Pooling all
derivative orders gives a compact set limit $K_{\eta,B}$ determined by
$(\eta,B)$; Figure~\ref{fig:cubic} shows an example whose limiting core
has interior points.

For the holes we view $(H_{N,k})_N$ as an array. Its normalized root
measures converge to $\eta_0=\frac12(\delta_{-1}+\delta_1)$, and its $k$
simple roots $x+i$, $x=3,\ldots,k+2$, are vanishing-mass locations. (For
each fixed $N$, the shadows of $H_{N,k}$ still involve all powers
$H_{N,k}^n$.) The selected transport paths are straight segments landing
on $[-1,1]$. As $x$ increases, the root moves to the right but its
landing point moves to the left, so the segments cross pairwise. Two transported
clusters never reach their crossing at a common ratio
(Figure~\ref{fig:trails}(b)), so the crossings appear only after all
derivative orders are pooled. They explain the limiting geometry, but
they do not by themselves prove that a fixed approximating core has a
hole: Hausdorff convergence to a graph with a cycle is compatible with
small gaps in every approximant. For each fixed polynomial we therefore
construct continuous paths $t\mapsto\gamma(t)$ in its fixed-ratio
supports, each following one transported cluster and ending at the
centroid. We force two such paths to intersect near each limiting
crossing and join them at the centroid into a loop in the core. The loop
winds nontrivially about a chosen point in a bounded face of the
limiting graph, and complete zero-set convergence keeps this point
outside the full shadow; hence it lies in a hole. Only finitely many
strict inequalities are used, so the conclusion persists under
coefficient perturbations.

A small \emph{positive} mass gives a complementary local picture. If a
location $b$ carries a small mass $\varepsilon$ instead of a vanishing
one, the fixed-ratio support contains, near each selected point
$\theta_b(t)$, an arc of size of order $\sqrt\varepsilon$; after
normalization, the distribution of mass along the arc converges to the
semicircle law. If these arcs are transverse to the transport path (an
explicit condition), then as the ratio varies they sweep out an open
\emph{support strip} containing the curve of their centers, which
therefore lies in the interior of the limiting core. This refinement is
not needed for the winding-number proof, but it explains the local
geometry visible in Figures~\ref{fig:hole} and~\ref{fig:small-mass}.

\subsection{Relation to earlier work and organization}
B\o gvad, H\"agg, and Shapiro \cite{BHS} study fixed-ratio asymptotics
for derivatives of polynomial powers through limiting potentials and
Cauchy transforms. The contour-selection and complete-convergence
framework for a fixed polynomial, together with the core--transient
decomposition, is developed in \cite{HS1}. Results imported from that
paper are restated in the normalization and notation needed here.
We cite their proofs where the arguments carry over, and explain the
extensions needed for real masses and varying degrees. The present
work begins where a purely measure-level description loses information:
at root locations whose normalized masses vanish.

Mart\'inez-Finkelshtein and Rakhmanov \cite{MFR} treat exterior
Cauchy-transform evolution for general arrays. Hall and Perales \cite{HP}
study deterministic arrays with asymptotically radial root distributions,
and Hall, Ho, Jalowy, and Kabluchko \cite{HHJK} study repeated
differentiation and related operators for random polynomials. Our
finite-location results retain information invisible to a limiting
probability measure: the exterior locations and exact multiplicities of
vanishing-mass clusters, together with the topology they generate after
all derivative orders are pooled.

Sections~\ref{sec:weighted}--\ref{sec:tame} also provide foundations
for the companion manuscripts. The real-weight and finite-location
theory, including topology and, at fixed spatial scale, boundary-length
bounds controlled by the number of atoms, is used in the measure decomposition, reflection
rigidity, and area-stability theory of~\cite{HS3}; the definable support
and incidence families also enter the analysis of caustics and core
boundaries in~\cite{HS4}.

Section~\ref{sec:weighted} introduces the barrier and the selected
saddles of the limiting landscape, constructs the limiting measures
$\mu_{t,\eta}$ for arbitrary positive masses, and establishes the
uniform analytic estimate used for varying-degree arrays.
Section~\ref{sec:clusters} proves the complete cluster limits and the
conservation law for zero counts when $m_j/D_j\to0$.
Section~\ref{sec:tame} establishes definability, finite topology, and
continuous paths through the family of fixed-ratio supports.
Section~\ref{sec:holes} gives the four-location and many-hole
constructions. Section~\ref{sec:small-atoms} develops the
small-positive-mass support strips and the semicircle profile.
Section~\ref{sec:hole-questions} concludes with sharp topology questions.

\section{Weighted barriers and uniform analytic estimates}
\label{sec:weighted}

We begin with a finite atomic probability measure. This separates the
limiting geometry from any choice of integer multiplicities. We define
the barrier of its landscape and the selected saddles, prove a uniform
saddle-point estimate for varying-degree arrays, and construct the
fixed-ratio limits $\mu_{t,\eta}$ together with the selected inverse.
The contour, potential-compactness, and small-order coefficient tools
from~\cite{HS1} are restated in the present normalization with references
to their proofs. The varying-degree saddle estimate and the construction
for arbitrary positive real masses are developed here.

\subsection{Potentials, saddles, and selection}
For a compactly supported finite measure $\nu$, our conventions are
\[
 U_\nu(z)=\int\log|z-\xi|\dd\nu(\xi),\qquad
 C_\nu(z)=\int\frac{\dd\nu(\xi)}{z-\xi},\qquad
 \Delta U_\nu=2\pi\nu.
\]
Off the support, $2\partial_zU_\nu=C_\nu$, where
$\partial_z=(\partial_x-i\partial_y)/2$ for $z=x+iy$.
Let
\begin{equation}\label{eq:weighted-data}
 \eta=\sum_{a\in A}p_a\delta_a,\qquad
 p_a>0,\qquad \sum_{a\in A}p_a=1,
\end{equation}
with $A\subset\C$ finite and nonempty, and put $s=|A|$. Write
\[
 C_\eta(w)=\sum_{a\in A}\frac{p_a}{w-a},\qquad
 \mathfrak c_\eta=\sum_{a\in A}p_a a,\qquad
 \mathfrak h(t)=-t\log t-(1-t)\log(1-t).
\]
The notation $Z(C_\eta)$ means the finite zeros of this rational
function, away from its poles.

For $z\notin A$ and $0<t<1$, define the real landscape
\begin{equation}\label{eq:weighted-landscape}
 L_{\eta,z,t}(w)=\sum_{a\in A}p_a\log|w-a|-t\log|w-z|.
\end{equation}
On a simply connected neighborhood avoiding $A\cup\{z\}$, choose
holomorphic logarithms and write
\begin{equation}\label{eq:weighted-phase}
 \Phi_{\eta,z,t}(w)=\sum_{a\in A}p_a\log(w-a)-t\log(w-z).
\end{equation}
Thus $L_{\eta,z,t}=\Ree\Phi_{\eta,z,t}$. The landscape tends to
$+\infty$ at $z$ and at infinity, with respective strengths $t$ and
$1-t$; these are its \emph{positive ends}. Its \emph{negative ends}
are the points of $A$, where it tends to $-\infty$.

An \emph{ordinary saddle} is a critical point of the phase away from
these ends. Set
\begin{equation}\label{eq:weighted-reduced-polynomials}
 Q_\eta(w)=\prod_{a\in A}(w-a),\qquad
 R_\eta(w)=\sum_{a\in A}p_a\prod_{a'\ne a}(w-a').
\end{equation}
Then $C_\eta=R_\eta/Q_\eta$, the two polynomials have no common zero,
and the saddle equation is
\begin{equation}\label{eq:weighted-saddle-equation}
 (w-z)R_\eta(w)-tQ_\eta(w)=0.
\end{equation}
It has degree $s$ and leading coefficient $1-t$. Equivalently, a
saddle satisfies $F_{t,\eta}(w)=z$, where the rational map
\begin{equation}\label{eq:weighted-F}
 F_{t,\eta}(w)=w-\frac{t}{C_\eta(w)}
\end{equation}
is understood after cancellation of common factors.

The \emph{critical height} of a saddle $w_*$ is $L_{\eta,z,t}(w_*)$.
With derivatives taken in $w$, its \emph{phase order} is the integer $\kappa\ge2$ characterized by
\[
 \Phi'(w_*)=\cdots=\Phi^{(\kappa-1)}(w_*)=0,
 \qquad \Phi^{(\kappa)}(w_*)\ne0.
\]
The saddle is \emph{simple} when $\kappa=2$. Its multiplicity in
\eqref{eq:weighted-saddle-equation} is $\kappa-1$, so the sum of
these multiplicities over any collection of saddles is at most $s$.
A local holomorphic coordinate reduces the phase to its critical value
plus a nonzero multiple of the $\kappa$th power of that coordinate.
The alternating regions in which its real part is larger or smaller
than the critical height are called the \emph{higher} and \emph{lower
sectors}. For a simple saddle one may choose a coordinate $v$ with
$\Phi-\Phi(w_*)=v^2/2$; we call $v$ a \emph{Morse coordinate}.
An observation point $z\notin A$ is \emph{regular} if all its ordinary
saddles are simple and their critical heights are pairwise distinct.

To specify which saddles contribute, work on the Riemann sphere
$\C\cup\{\infty\}$ and set
\[
 E_\ell(\eta,z,t)=
 \{w\notin A\cup\{z\}:L_{\eta,z,t}(w)\ge\ell\}
 \cup\{z,\infty\}.
\]
The \emph{barrier level} is
\begin{equation}\label{eq:weighted-barrier-definition}
 \Blev_\eta(z,t)=\sup\{\ell\in\R:
 z\text{ and }\infty\text{ belong to the same connected component of }
 E_\ell(\eta,z,t)\}.
\end{equation}
At this level, let $U_z$ and $U_\infty$ be the strict-superlevel
components containing the corresponding positive ends. A
\emph{connecting saddle}, or \emph{selected saddle}, is an ordinary
saddle lying in $\overline{U_z}\cap\overline{U_\infty}$. Selection is a
connectivity condition at the barrier level; equality of critical
heights alone does not imply it. Figure~\ref{fig:selection} shows these
objects for $\eta_0=\frac12(\delta_{-1}+\delta_1)$.

\begin{figure}[htbp]
\centering
\begin{minipage}[b]{0.4863\textwidth}\centering
 \includegraphics[width=\linewidth]{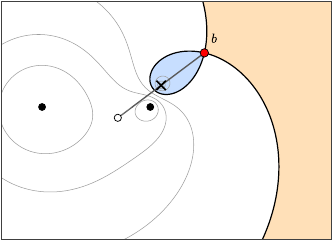}\\[1pt]
 {\small (a)}
\end{minipage}\hfill
\begin{minipage}[b]{0.4863\textwidth}\centering
 \includegraphics[width=\linewidth]{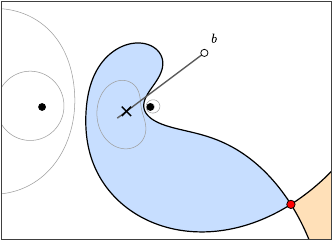}\\[1pt]
 {\small (b)}
\end{minipage}
\caption{The landscape
$L_{\eta_0,z,t}(w)=\frac12\log|w^2-1|-t\log|w-z|$ at the transported
point $z=\theta_b(t)$ (cross) of the vanishing-mass location $b=2+i$;
black disks: $\pm1$. Blue and orange: the components $U_z$ and
$U_\infty$ of the region above the barrier; black: the level curve
$L_{\eta_0,z,t}=\Blev_{\eta_0}(z,t)$; gray: further level lines. Red:
the connecting saddle; white: the other saddle. Gray segment: the
transport path $\theta_b$, dotted where $b$ is not selected.
(a)~$t=\frac12$: the connecting saddle is $b$. (b)~$t=\frac9{10}$, after
$\theta_b$ has reached the real axis at $t=\frac56$: the connecting
saddle is $9/b$, and $b$ lies below the barrier.}
\label{fig:selection}
\end{figure}

For a positively oriented piecewise smooth Jordan curve $\gamma$
avoiding $z$, put
\[
 \wind(\gamma,z)=\frac1{2\pi i}\oint_\gamma\frac{\dd w}{w-z}.
\]
A \emph{barrier contour} is such a curve, avoiding $A\cup\{z\}$,
winding once about $z$, passing once through every connecting saddle,
and lying strictly below the barrier elsewhere. A \emph{one-saddle
separating contour} is a barrier contour whose only connecting saddle is simple, with a
nondegenerate maximum there. On the complement of any
fixed small saddle neighborhood, its height is at most
$\Blev_\eta(z,t)-\delta$ for some $\delta>0$. This is the
\emph{strict height gap} used in the integral estimates.

\begin{proposition}[Contour facts for positive atomic weights]
\label{prop:weighted-contours}
\cite[Lemmas~2.1--2.2 and~2.4--2.9]{HS1}
The barrier is finite and continuous in the observation point, the
interior ratio, and the positive mass and distinct-location data.
It satisfies the min--max identity
\begin{equation}\label{eq:weighted-minmax}
 \Blev_\eta(z,t)=\inf_\gamma\max_{w\in\gamma}L_{\eta,z,t}(w),
\end{equation}
where $\gamma$ ranges over positively oriented piecewise smooth
Jordan curves avoiding $A\cup\{z\}$ and winding once about $z$.
There is a barrier contour through all connecting saddles, with no
restriction on their finite phase orders. Through each connecting
saddle there is also a path to the two positive ends whose unique
minimum is that saddle.

For fixed $\eta,t$, almost every observation point is regular. At a
regular point there is a locally persistent one-saddle separating
contour. The same conclusion holds at every point of a disk on which
$\Blev_\eta(\cdot,t)$ is harmonic, even if other, unselected saddles
are degenerate or have equal heights.
\end{proposition}
\begin{proof}
The proofs of the cited lemmas apply to $L_{\eta,z,t}$ without an
integrality assumption on the logarithmic weights. They use only the
positive end strengths $t$ and $1-t$, the negative end strengths $p_a$,
local holomorphic phases, and finiteness of the reduced saddle equation
\eqref{eq:weighted-saddle-equation}. Thus the separation, min--max,
contour, regularity, and harmonic-barrier arguments carry over in the
normalization used here.

For joint continuity in the masses, use the two bounds in the proof of
\cite[Lemma~2.2]{HS1}. A nearly minimizing contour and the compact part
of a path joining the positive ends avoid all singularities, so the
landscape varies uniformly on them under small changes of $z$, $t$,
the positive masses, and the distinct locations. The truncated end
pieces can be reattached uniformly because the strengths $t$ and $1-t$
remain positive. The same upper contour bound and lower path bound
therefore give continuity in all the stated parameters.
\end{proof}

\subsection{A uniform estimate for varying-degree arrays}
The contours of Proposition~\ref{prop:weighted-contours} allow
connecting saddles of any finite phase order, whereas a Gaussian
asymptotic requires a unique selected saddle of phase order two. The
following lemma gives such an asymptotic for phases that vary with $j$.
It is used both for integer approximations of real weights
(Proposition~\ref{prop:weighted-foundations}) and for root arrays
(Section~\ref{sec:clusters}); two auxiliary lemmas follow.

\begin{lemma}[Uniform one-saddle transfer]\label{lem:uniform-transfer}
\cite[cf. Proposition~3.1]{HS1}
Let $H_j$ be monic polynomials of degrees $D_j\to\infty$. On fixed
neighborhoods of the simple saddles of the contours described below, assume
that local branches
\[
 V_j(w)=D_j^{-1}\log H_j(w)
\]
converge, after choices of additive imaginary constants, to a
holomorphic phase $V(w)$ with all local derivatives. On fixed
neighborhoods of the remaining contour pieces assume uniform convergence
of their real parts on compact subsets. The neighborhoods are root-free
for all sufficiently large $j$.

Suppose that, for parameters $(z,t)$ in a compact neighborhood in
$\C\times(0,1)$, the limiting phase
$V(w)-t\log(w-z)$ has a persistent one-saddle separating contour,
with saddle $W(z,t)$. Assume its saddle is uniformly nondegenerate and
its remaining contour pieces have a uniform strict height gap. Then
there are nearby simple saddles $W_j(z,t)$ and locally nonvanishing
amplitudes $\mathfrak g_j(z,t)$ such that, at every exact ratio
$t=m/D_j$ in a smaller compact parameter neighborhood,
\begin{equation}\label{eq:uniform-transfer}
 \frac{H_j^{(m)}(z)}{m!}
 =\frac{H_j(W_j(z,t))}{\sqrt{D_j}\,(W_j(z,t)-z)^m}
       \mathfrak g_j(z,t)\bigl(1+O(D_j^{-1})\bigr).
\end{equation}
The error is uniform in $z$ and in the exact ratio. The saddle and
amplitude are holomorphic in $z$ and real-analytic in $t$ on each
local parameter chart. With
\[
 C_j(w)=\frac{H_j'(w)}{D_jH_j(w)},\qquad
 F_{j,t}(w)=w-\frac{t}{C_j(w)},
\]
the oriented Gaussian amplitude satisfies
\begin{equation}\label{eq:transfer-amplitude}
 \mathfrak g_j(z,t)^2=
 \frac{1}{2\pi t F_{j,t}'(W_j(z,t))}.
\end{equation}
No rate of convergence of $V_j$ to $V$ is required.
\end{lemma}
\begin{proof}
The implicit-function theorem gives $W_j\to W$ with derivatives on
smaller parameter neighborhoods. Uniform nondegeneracy provides Morse
coordinates for the \emph{moving} phases
\[
 \Phi_j(w,z,t)=V_j(w)-t\log(w-z)
\]
on a common smaller saddle disk. Replace the local contour passage by
the descending axis of this coordinate. The rest of the contour can
be joined to this passage within the fixed root-free neighborhoods.
Uniform convergence of real parts preserves, for example, half the
limiting height gap. After a finite parameter cover, the contour
lengths, analytic derivatives, inverse-coordinate derivatives, and
saddle nondegeneracy constants are uniformly bounded.

At $t=m/D_j$, Cauchy's formula is
\[
 \frac{H_j^{(m)}(z)}{m!}
 =\frac1{2\pi i}\oint
 \frac{H_j(w)}{(w-z)^{m+1}}\,\dd w.
\]
This kernel is single-valued. On the saddle chart, it has the exact
centered-phase expression
\[
 \frac{H_j(w)}{(w-z)^{m+1}}
 =\frac{H_j(W_j)}{(W_j-z)^m}
   \frac{e^{D_j(\Phi_j(w,z,t)-\Phi_j(W_j,z,t))}}{w-z}.
\]
The additive constants in the local logarithms cancel here. The
part outside the saddle disk is exponentially smaller than the saddle
contribution, uniformly in the parameters. Inside, the Morse
coordinate reduces the exponential to a Gaussian. Expand the analytic
amplitude on a symmetric local passage. The odd term integrates to
zero, and the next term has relative size $O(D_j^{-1})$; analytic
bounds control the remaining terms uniformly. The leading term is
nonzero because both $W_j-z$ and $\partial_w^2\Phi_j(W_j,z,t)$ are
nonzero. This proves \eqref{eq:uniform-transfer}.

At the saddle,
\[
 (W_j-z)^2\partial_w^2\Phi_j(W_j,z,t)
       =t F_{j,t}'(W_j).
\]
The factor $1/(2\pi i)$ and the orientation of the descending passage
give \eqref{eq:transfer-amplitude}; the contour fixes the local square
root, up to a constant sign. Analytic parameter dependence follows from
the same construction. Choose the physical parameter charts to be small
balls with closures inside the original charts. On their overlaps the
nearby saddle branches agree by local uniqueness, and hence their
Gaussian squares agree. For large $j$, an exact ratio in each nonempty
overlap fixes the sign through Cauchy's formula and the relative error;
continuity identifies the oriented amplitudes on the whole overlap.
A finite cover gives one threshold and one error constant. The exact moving phase remains in the leading
exponential throughout: it is never replaced there by its limit.
Consequently the proof imposes no condition on
$D_j\|V_j-V\|$ or on a discrepancy between $m$ and a prescribed
limiting ratio times $D_j$.
\end{proof}

The lemma is local in the observation parameter. It does not assert
that an arbitrary nonzero amplitude has zero winding on an annulus.
When exact cluster counts require that additional conclusion, we will
verify it from the nonvanishing limiting radicand in
\eqref{eq:transfer-amplitude}.

\begin{lemma}[Compactness of logarithmic potentials]
\label{lem:potential-compactness}
\cite[Lemma~3.2]{HS1}
If probability measures $\nu_j$ are supported in one compact set and
$\nu_j\Longrightarrow\nu$, then
$U_{\nu_j}\to U_\nu$ in $L^1_{\mathrm{loc}}(\C)$. Their potentials
are uniformly integrable on each bounded planar set. The same assertion
holds for any family with that common support whose closure is weakly
compact.
\end{lemma}

\begin{lemma}[Small-order coefficients]
\label{lem:small-order-coefficients}
\cite[Lemma~4.1]{HS1}
Let $\mathscr A_z(u)$ be holomorphic for $z$ near a compact set $X$
and $|u|<r$. Suppose that $|\mathscr A_z(0)|$ and
\[
 |\mathfrak b_z|,\qquad
 \mathfrak b_z:=\frac{\partial_u\mathscr A_z(0)}{\mathscr A_z(0)},
\]
are bounded below by positive constants on $X$. If $k_n\ge1$ and
$k_n/n\to0$, then
\[
 [u^{k_n}]\mathscr A_z(u)^n\ne0\qquad(z\in X)
\]
for all sufficiently large $n$. The same conclusion holds for varying
functions $\mathscr A_{n,z}$ that are holomorphic and uniformly bounded
on a common neighborhood of $X\times\{0\}$, provided both lower bounds
are uniform in $n$. On a smaller parameter neighborhood,
\begin{equation}\label{eq:small-coeff-logder}
 \partial_z\log\bigl([u^{k_n}]\mathscr A_{n,z}(u)^n\bigr)
 =n\,\partial_z\log\mathscr A_{n,z}(0)+O(k_n)
\end{equation}
holds uniformly. Here $\partial_z\log f$ means $f'/f$, so no global
logarithm is required; for $k_n=0$ the identity is exact.

\end{lemma}

\subsection{The limiting measures and their selected inverse}
Write $(u)_+=\max\{u,0\}$. The next proposition constructs all the
weighted objects used in Theorem~\ref{thm:finite-cluster-universality}.

\begin{proposition}[Shadows of a finite atomic measure]
\label{prop:weighted-foundations}
For each $0<t<1$ there is a unique probability measure $\mu_{t,\eta}$
whose logarithmic potential satisfies
\begin{equation}\label{eq:weighted-potential}
 U_{\mu_{t,\eta}}(z)=
 \frac{\Blev_\eta(z,t)-\mathfrak h(t)}{1-t},\qquad z\notin A.
\end{equation}
These measures are weakly continuous in $t$. Their fixed-ratio
supports, pooled measure, and pooled support are
\begin{equation}\label{eq:weighted-shadows}
 \begin{split}
 S_{t,\eta}&=\supp\mu_{t,\eta},\qquad
 \rho_\eta=2\int_0^1(1-t)\mu_{t,\eta}\dd t,\\
 S_\eta&=\supp\rho_\eta
       =\overline{\bigcup_{0<t<1}S_{t,\eta}}.
 \end{split}
\end{equation}
All these supports are compact subsets of $\conv A$. The atomic parts
are exactly
\begin{equation}\label{eq:weighted-atoms}
 \mu_{t,\eta}^{\mathrm{at}}
   =\sum_{a\in A}\frac{(p_a-t)_+}{1-t}\delta_a,
 \qquad
 \rho_\eta^{\mathrm{at}}=\sum_{a\in A}p_a^2\delta_a.
\end{equation}
Every $a\in A$ lies in $S_{t,\eta}$ for $0<t\le p_a$, with $t<1$.
The supports are Hausdorff-continuous in $t$, and
\begin{equation}\label{eq:weighted-endpoints}
 S_{t,\eta}\longrightarrow A\cup Z(C_\eta)\quad(t\downarrow0),
 \qquad
 S_{t,\eta}\longrightarrow\{\mathfrak c_\eta\}\quad(t\uparrow1).
\end{equation}

On $\C\setminus S_{t,\eta}$, the Cauchy transform $C_{\mu_{t,\eta}}$
does not vanish.
The \emph{selected inverse} is the holomorphic function
\begin{equation}\label{eq:weighted-W}
 W_{t,\eta}(z)=z+\frac{t}{(1-t)C_{\mu_{t,\eta}}(z)}.
\end{equation}
It is locally real-analytic in $t$ and satisfies
\begin{equation}\label{eq:weighted-inverse}
 \begin{split}
 F_{t,\eta}(W_{t,\eta}(z))&=z,\\
 C_\eta(W_{t,\eta}(z))&=(1-t)C_{\mu_{t,\eta}}(z),\\
 F_{t,\eta}'(W_{t,\eta}(z))&\ne0.
 \end{split}
\end{equation}
At a nonroot observation point it is the unique selected saddle. If
$a\in A\setminus S_{t,\eta}$, then $t>p_a$; the inverse identities
extend holomorphically across the observation point $z=a$, and there is
a locally persistent one-saddle separating contour there.

For polynomial root data, put
$\FixSh_{t,P}=S_{t,\rootmu_P}$ and $\rho_P=\rho_{\rootmu_P}$.
Then
\begin{equation}\label{eq:polynomial-weighted-identification}
 S_{\rootmu_P}=\Sh_P,
 \qquad
 F_{t,\rootmu_P}(w)=w-dt\frac{P(w)}{P'(w)}.
\end{equation}
Moreover, $\mu_{t,\rootmu_P}$ is the fixed-ratio limit for derivatives of
$P^n$ with $m/(dn)\to t$. Accordingly, we call $S_\eta$ the
\emph{limiting core} of $\eta$.
\end{proposition}
\begin{proof}
\emph{Construction and normalization.}
Choose positive integers $M_{a,j}$ with
$\sum_aM_{a,j}=D_j\to\infty$ and $M_{a,j}/D_j\to p_a$, and put
\[
 G_j(w)=\prod_{a\in A}(w-a)^{M_{a,j}}.
\]
On every compact set avoiding $A$, the real normalized logarithms
converge to $U_\eta$, and all derivatives of their local holomorphic
branches converge to those of $\sum_a p_a\log(w-a)$. At a regular
observation point, Proposition~\ref{prop:weighted-contours} supplies
a one-saddle separating contour. Its compact pieces avoid $A$, and its strict
height gap and simple saddle satisfy every hypothesis of
Lemma~\ref{lem:uniform-transfer}.

For integers $m_j$ with $m_j/D_j\to t$, the lemma gives the limit
$\Blev_\eta(z,t)$ of $D_j^{-1}\log|G_j^{(m_j)}(z)/m_j!|$ at almost
every observation point $z$. The leading coefficient
of $G_j^{(m_j)}/m_j!$ is $\binom{D_j}{m_j}$, with
\[
 \frac1{D_j}\log\binom{D_j}{m_j}\longrightarrow\mathfrak h(t).
\]
Lemma~\ref{lem:potential-compactness} gives subsequential weak
limits of the normalized root measures of $G_j^{(m_j)}$ and local
$L^1$ convergence of their logarithmic potentials.
The transfer formula identifies every such limit almost everywhere with
the right-hand side of \eqref{eq:weighted-potential}. Both sides have
their canonical subharmonic representatives, and the barrier is
continuous off $A$, so the identity holds at every nonroot point.
Gauss--Lucas supplies the common compact support $\conv A$. Applying
$\Delta/(2\pi)$ to the potential proves uniqueness and independence of
the integer approximation. For a sequence of ratios tending to $t$,
the same compactness argument and barrier continuity identify every
subsequential limit with $\mu_{t,\eta}$; hence the measures are weakly
continuous in $t$. The estimates are local and uniform while all masses
remain positive and the limiting locations remain distinct.

\emph{Atoms and thresholds.}
Fix $a\in A$. If $t<p_a$, then for every integer approximation with
$m_j/D_j\to t$ the zero $a$ has multiplicity exactly
$M_{a,j}-m_j$ in $G_j^{(m_j)}$. This forces an atom of mass at least
$(p_a-t)/(1-t)$. For the reverse bound, put
$\widehat C=(1-t)C_{\mu_{t,\eta}}$. The exact identity
$Q_\eta G_j'=R_jG_j$, where
$R_j=\sum_aM_{a,j}\prod_{b\ne a}(w-b)$, and its differentiated
Leibniz expansion give, in the weak derivative limit,
\begin{equation}\label{eq:weighted-cleared-cauchy}
 \prod_{b\in A}\bigl((z-b)\widehat C(z)+t\bigr)
 =\sum_{b\in A}p_b
       \prod_{c\in A\setminus\{b\}}
             \bigl((z-c)\widehat C(z)+t\bigr)
 \quad\text{for almost every }z.
\end{equation}
Here truncation of the Cauchy kernel gives local $L^p$ convergence
for $1\le p<2$, just as truncation of the logarithmic kernel gives
Lemma~\ref{lem:potential-compactness}.

If $q=\mu_{t,\eta}(\{a\})>0$, shrinking dilations give
$rC_{\mu_{t,\eta}}(a+ru)\to q/u$ in local $L^p$. Choose
$u\ne0$ and a sequence $r\downarrow0$ for which this limit and
\eqref{eq:weighted-cleared-cauchy} hold pointwise. Multiply every
factor belonging to a site other than $a$ by $r$ and pass to the
limit. Their limiting product is nonzero, and cancellation gives
\[
 t+(1-t)q=p_a.
\]
Together with the forced lower bound this gives exactly
$q=(p_a-t)_+/(1-t)$, including zero mass at and beyond the
threshold. The finite potential off $A$ excludes all other atoms,
proving the first formula in \eqref{eq:weighted-atoms}.

At the threshold $t=p_a<1$, the disk-mass argument in the proof of
\cite[Lemma~3.7]{HS1} applies with $(d-\alpha)\mu_{\alpha,P}$ replaced
by $(1-t)\mu_{t,\eta}$. Its inputs are the atom formula just proved and
the persistence of one-saddle contours on a hypothetical support-free
annulus, supplied here by Proposition~\ref{prop:weighted-contours}.
It follows that $a\in S_{p_a,\eta}$.

\emph{The inverse and exclusion off the support.}
As in the proof of \cite[Proposition~3.5]{HS1}, on a root-free disk
disjoint from $S_{t,\eta}$ the potential formula makes the barrier
harmonic. Proposition~\ref{prop:weighted-contours}
therefore supplies one simple connecting saddle $W$, and spatial
differentiation gives
\[
 (1-t)C_{\mu_{t,\eta}}(z)=\frac{t}{W-z}=C_\eta(W).
\]
This proves \eqref{eq:weighted-W} and \eqref{eq:weighted-inverse}
there. To make the algebra and its nonvanishing consequence explicit,
recall $\widehat C=(1-t)C_{\mu_{t,\eta}}$ and define the
\emph{reduced symbol}
\begin{equation}\label{eq:weighted-reduced-symbol}
 \mathscr P_{t,\eta}(z,\zeta)=
 \zeta^sQ_\eta(z+t/\zeta)
       -\zeta^{s-1}R_\eta(z+t/\zeta).
\end{equation}
After expansion, this is a polynomial in $\zeta$. It satisfies
\[
 \mathscr P_{t,\eta}(z,\widehat C(z))=0,
 \qquad \mathscr P_{t,\eta}(z,0)=t^{s-1}(t-1)\ne0.
\]
Consequently $\widehat C$ cannot vanish off the support, including
at a root by holomorphic continuation.

If $a\in A$ is outside the support, the threshold assertion gives
$t>p_a$. At observation point $a$ the landscape becomes
\[
 (p_a-t)\log|w-a|+\sum_{a'\ne a}p_{a'}\log|w-a'|.
\]
The combined singularity at $a$ is now a positive end, as is infinity.
After deleting the atom at $a$ and normalizing the remaining masses,
this is $1-p_a$ times a nonroot landscape with ratio
$(t-p_a)/(1-p_a)\in(0,1)$.
The coalesced-end argument in the proof of
\cite[Lemma~3.8]{HS1} applies here: its contour and path bounds use
only the strict inequality $t>p_a$ and local holomorphic phases, so
they do not require integer logarithmic weights. It identifies this
barrier with the continuous extension from nearby nonroot observation
points. The harmonic-barrier argument in the same proof then gives a
unique simple connecting saddle and a persistent one-saddle separating
contour at $a$. Thus the inverse identities extend holomorphically
across $a$, with $F_{t,\eta}'(W(z))W'(z)=1$.
For analytic dependence across an observation point $a\in A\setminus S_{t,\eta}$, take the
simple implicit inverse branch on a small disk and identify it on a
surrounding nonroot circle. Its Cauchy model is holomorphic in the disk;
agreement on the circle forces the enclosed mass to vanish, and Cauchy's
integral formula extends agreement throughout the disk. This gives a
common off-support neighborhood and local real-analyticity in $t$.

\emph{Hausdorff continuity and endpoints.}
Weak convergence supplies the lower support estimate for the integer
approximations $G_j^{(m_j)}$. At nonroot points outside $S_{t,\eta}$,
the preceding contours and Lemma~\ref{lem:uniform-transfer} give
uniform zero exclusion. At an atom location $a\in A$ outside the support,
construct the local holomorphic polynomial saddle and nonvanishing Gaussian
amplitude on a whole observation disk. On a surrounding nonroot circle,
local uniqueness identifies its saddle with the transfer saddle; the exact
radicands identify the squares of their amplitudes. The squared transfer
formula therefore has uniformly vanishing relative error against the
nonvanishing holomorphic model on the circle. The maximum modulus
principle extends this error bound through the disk and excludes all
derivative zeros there. Thus
\[
 Z(G_j^{(m_j)})\longrightarrow S_{t,\eta}
 \quad\text{whenever }m_j/D_j\to t\in(0,1).
\]
The same weak lower estimate and persistent-contour upper estimate,
or a diagonal use of this convergence, give Hausdorff continuity in $t$.

For endpoint exclusion we use Lemma~\ref{lem:small-order-coefficients}.
On compact observation
sets avoiding $A\cup Z(C_\eta)$, the functions
\[
 \exp\left(\frac1{D_j}\log\frac{G_j(z+u)}{G_j(z)}\right)
\]
are holomorphic and uniformly bounded on a common small $u$-disk,
have constant term one, and have linear coefficient tending to
$C_\eta(z)\ne0$. The coefficient of order $m_j$ of their $D_j$th
powers is $G_j^{(m_j)}(z)/(m_j!G_j(z))$. Hence there are no zeros
there when $m_j=o(D_j)$. For the other endpoint use
\[
 \exp\left(\frac1{D_j}\sum_{a\in A}M_{a,j}
                              \log(1+(z-a)u)\right).
\]
Its linear coefficient tends to $z-\mathfrak c_\eta$, and the
coefficient of order $D_j-m_j$ of its $D_j$th power is
$G_j^{(m_j)}(z)/m_j!$. This excludes zeros away from the centroid
when $D_j-m_j=o(D_j)$. A diagonal choice of integer approximations
at ratios tending to either endpoint gives the upper estimates in
\eqref{eq:weighted-endpoints}.

For the lower endpoint estimate, follow the proof of
\cite[Theorem~4.5]{HS1}. As $t\downarrow0$, the atom formula gives
$\mu_{t,\eta}\Longrightarrow\eta$ and supplies every point of $A$
in the lower set limit. A disk about a zero of $C_\eta$ cannot be
missed along a sequence of supports: on a smaller such disk the
nonvanishing holomorphic Cauchy transforms would converge to
$C_\eta$, contradicting Hurwitz's theorem. The upper endpoint limit
is a singleton, so nonemptiness gives its lower estimate.

Finally, weak continuity implies the closed-union identity in
\eqref{eq:weighted-shadows}. Integrating the atomic contribution gives
$2\int_0^{p_a}(p_a-t)\dd t=p_a^2$ and no other atoms.

For $\eta=\rootmu_P$, take the exact integer approximation
$G_n=c^{-n}P^n$, where $c$ is the leading coefficient of $P$.
Uniqueness of the fixed-ratio limit and \cite[Theorem~3.3]{HS1}, with
$\alpha=dt$, identify the measures constructed here with those of
Paper~I. The pooled-support identity $S_{\rootmu_P}=\Sh_P$ is then
\cite[Theorem~5.1]{HS1}. Finally, $C_{\rootmu_P}=P'/(dP)$ gives the
rational-map identity in \eqref{eq:polynomial-weighted-identification}.
The degree-one case follows directly by differentiation.
\end{proof}
\section{Complete limits for finite root clusters}
\label{sec:clusters}

The weighted measures retain the positive limiting masses. To recover
all zero locations, we must also retain the limiting root set. We
first prove the complete fixed-ratio and pooled limits, including an
exact count near each exterior transported point. We then give the
integer conservation law at the lower endpoint.

\subsection{The limiting data and the transport condition}
Let $H_j$ be monic of degree $D_j\to\infty$, and suppose
\begin{equation}\label{eq:finite-location-input}
 \eta_j:=\frac1{D_j}\divz H_j\Longrightarrow\eta,
 \qquad d_{\mathrm H}(Z(H_j),B)\longrightarrow0,
\end{equation}
where $B$ is finite and nonempty and $\eta$ is as in
\eqref{eq:weighted-data}. Necessarily $A\subseteq B$.
A point $b\in B\setminus A$ represents a \emph{vanishing-mass
cluster}: every sufficiently small fixed disk about $b$ eventually
contains an input root, but the fraction of input roots in the disk
tends to zero. Neither the number nor the multiplicities of these roots
are assumed bounded, and no lower bound on their separation is imposed.

For $b\in B\setminus A$ with $C_\eta(b)\ne0$, define the straight
transport path
\begin{equation}\label{eq:transport-path}
 \theta_b(t):=F_{t,\eta}(b)=b-\frac{t}{C_\eta(b)},\qquad 0\le t\le1.
\end{equation}
A point of this path contributes an exterior zero, that is, a zero
outside $S_{t,\eta}$, only when the selected inverse $W_{t,\eta}$ of
Proposition~\ref{prop:weighted-foundations} maps it back to $b$.
Accordingly, set
\begin{equation}\label{eq:selected-outliers}
 O_t(\eta,B):=\left\{\theta_b(t):
 \begin{array}{l}
 b\in B\setminus A,\ C_\eta(b)\ne0,\\
 \theta_b(t)\notin S_{t,\eta},\quad
 W_{t,\eta}(\theta_b(t))=b
 \end{array}\right\}.
\end{equation}
The equality involving $W_{t,\eta}$ fixes the inverse branch; the
saddle equation alone would not do so. Thus $O_t(\eta,B)$ is determined
by the limiting data, not by an extra hypothesis on the array. When
$\theta_b(t)\in O_t(\eta,B)$, we say that $b$ is \emph{selected} at
ratio $t$; Figure~\ref{fig:selection} shows a selected and an
unselected ratio for $\eta_0$ and $b=2+i$. Set
\begin{equation}\label{eq:complete-weighted-shadow}
 K_{\eta,B}:=S_\eta\cup B\cup
       \overline{\bigcup_{0<t<1}O_t(\eta,B)}.
\end{equation}
It is compact because the finitely many transport paths lie in bounded
segments.

\begin{theorem}[Complete limits for finite root clusters]\label{thm:finite-cluster-universality}
Assume \eqref{eq:finite-location-input}, and let $0\le m_j<D_j$ with
$m_j/D_j\to t\in(0,1)$. Then
\begin{align}
 \frac1{D_j-m_j}\divz H_j^{(m_j)}
      &\Longrightarrow\mu_{t,\eta},\label{eq:finite-cluster-measure}\\
 Z(H_j^{(m_j)})&\longrightarrow
          S_{t,\eta}\cup O_t(\eta,B)
          \quad\text{in Hausdorff distance}.\label{eq:finite-cluster-fixed-set}
\end{align}
For the pooled measure and zero set in
\eqref{eq:finite-order-objects},
\begin{equation}\label{eq:finite-cluster-pooled}
 \pi_{H_j}\Longrightarrow\rho_\eta,
 \qquad \mathcal Z(H_j)\longrightarrow K_{\eta,B}
 \quad\text{in Hausdorff distance}.
\end{equation}
There is also exact local conservation of multiplicity. For
$z_0=\theta_b(t)\in O_t(\eta,B)$, choose a sufficiently small fixed disk $\mathcal D$ about $z_0$ and a fixed
disk $\mathcal D_b$ about $b$ with $\overline{\mathcal D_b}\subset W_{t,\eta}(\mathcal D)$
and $\mathcal D_b\cap B=\{b\}$. Then, for all sufficiently large $j$,
\begin{equation}\label{eq:outlier-count}
 (\divz H_j^{(m_j)})(\mathcal D)=(\divz H_j)(\mathcal D_b).
\end{equation}
Thus an exterior cluster retains its exact number of zeros, even when
that number is unbounded but sublinear in $D_j$.
\end{theorem}

Figure~\ref{fig:trails} illustrates the theorem for the polynomials
$P_N$ of Theorem~\ref{thm:intro-holes}, for which $K_{\eta,B}$ is
determined in Lemma~\ref{lem:quadratic-spikes}, and
Figure~\ref{fig:cubic} for a configuration whose limiting core
$S_\eta$ has interior points.

\begin{figure}[htbp]
\centering
\includegraphics[width=\textwidth]{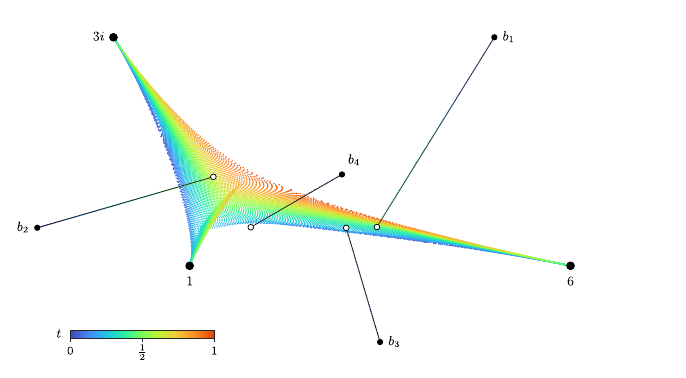}
\caption{All zeros of all derivatives of
$H(z)=((z-1)(z-6)(z-3i))^{70}(z-b_1)(z-b_2)(z-b_3)(z-b_4)$, colored by
$t=m/D$, $D=214$. The limiting core $S_\eta$ of
$\eta=\frac13(\delta_1+\delta_6+\delta_{3i})$ is the curved triangle
filled by the zeros, and $b_1=5+3i$, $b_2=-1+\frac i2$,
$b_3=\frac72-i$, $b_4=3+\frac65i$ are vanishing-mass locations. Dark
lines: the selected parts of their transport paths, which end in $S_\eta$
(circles). A trail may enter $S_\eta$, since it only has to avoid the
fixed-ratio support $S_{t,\eta}$ at its own ratio $t$.}
\label{fig:cubic}
\end{figure}

\begin{proof}
All input roots lie in a common compact set. Away from $B$,
\begin{equation}\label{eq:cluster-log-convergence}
 \frac1{D_j}\log|H_j(w)|\longrightarrow U_\eta(w),\qquad
 \frac{H_j'(w)}{D_jH_j(w)}\longrightarrow C_\eta(w),
\end{equation}
locally uniformly, with all derivatives of the second expression.
To see this, place disjoint small disks around $B$. They eventually
contain every root; weak convergence gives their limiting masses, and
the logarithmic kernels and their derivatives are uniformly bounded and
continuous away from the disks. Shrink the disks to obtain
\eqref{eq:cluster-log-convergence}.

\smallskip\noindent\emph{1. Weak convergence and exterior exclusion.}
At a regular observation point of the limiting landscape, take a
separating contour. Unless its selected saddle belongs to $B$, the
saddle neighborhood avoids $B$. Deform the compact remainder slightly
to avoid the other points of $B$, preserving its winding number and
strict height gap. Equation~\eqref{eq:cluster-log-convergence} then
verifies the hypotheses of Lemma~\ref{lem:uniform-transfer}: the
nearby saddle is uniformly simple and the remainder retains its strict
height gap. The resulting Gaussian contribution is nonzero. A selected saddle can meet
a fixed $b\in B\setminus A$ only at $z=F_{t,\eta}(b)$ with
$C_\eta(b)\ne0$. These finitely many points do not affect
almost-everywhere potential convergence. Subtracting
$D_j^{-1}\log\binom{D_j}{m_j}\to \mathfrak h(t)$ and using potential
compactness proves \eqref{eq:finite-cluster-measure}.
For a continuous test function, this convergence is uniform over
compact interior ratio intervals: a contrary sequence of ratios has
a convergent subsequence, to which the same argument applies.
The weights in \eqref{eq:finite-order-objects} form Riemann sums for
$2(1-t)\dd t$. The remaining endpoint ranges have arbitrarily small
total weight. This proves the pooled weak limit.

The transfer estimate is also uniform near any $z\notin S_{t,\eta}$ for which
$W_{t,\eta}(z)\notin B$. At an atom location $a\in A$ outside the support, take a small observation
disk on which the selected inverse avoids $B$. The local phases of $H_j$ on this disk construct a
holomorphic nonvanishing saddle model. On a surrounding nonroot circle,
the transfer formula has the same saddle and squared amplitude as this
model; the maximum modulus principle propagates the squared relative
error through the disk and excludes zeros there. The reduced saddle
equation prevents the selected inverse from meeting $A$.
Every exterior limit of zeros must therefore belong to $O_t(\eta,B)$,
which is the upper estimate in \eqref{eq:finite-cluster-fixed-set}.

\smallskip\noindent\emph{2. Exact counting at a transported cluster.}
Fix $z_0=F_{t,\eta}(b)\in O_t(\eta,B)$. On a small closed disk
$\overline{\mathcal D}$ about $z_0$, the selected inverse $W=W_{t,\eta}$ is
holomorphic and one-to-one, $W-z$ is nonzero, and $W(\overline{\mathcal D})$
avoids $A$ and every point of $B$ except $b\in W(\mathcal D)$. Choose $\mathcal D_b$
with closure in $W(\mathcal D)$ and containing no other point of $B$.

On an annulus about $\partial\mathcal D$, the equation
\[
 \frac{H_j'(w)}{D_jH_j(w)}-\frac{t_j}{w-z}=0,
 \qquad t_j:=m_j/D_j,
\]
has a single-valued simple branch $W_j(z)$ converging to $W(z)$,
with derivatives. This follows from the implicit function theorem and
\eqref{eq:cluster-log-convergence}, since $W(\partial\mathcal D)$ avoids $B$.
A finite cover of the annulus by the parameter charts of
Lemma~\ref{lem:uniform-transfer} gives
\begin{equation}\label{eq:outlier-annulus-asymptotic}
 \frac{H_j^{(m_j)}(z)}{m_j!}
  =\frac{1}{\sqrt{D_j}}
      \frac{H_j(W_j(z))}{(W_j(z)-z)^{m_j}}
       \mathfrak g_j(z)\bigl(1+O(D_j^{-1})\bigr).
\end{equation}
The amplitude $\mathfrak g_j$ is holomorphic, nonzero, and has winding number
zero on $\partial\mathcal D$. For the last assertion, put locally
\[
 C_j(w)=\frac{H_j'(w)}{D_jH_j(w)},\qquad
 F_j(w)=w-\frac{t_j}{C_j(w)},\qquad
 \phi_j(w,z)=D_j^{-1}\log H_j(w)-t_j\log(w-z).
\]
The saddle equation gives
\[
 (W_j-z)^2\partial_w^2\phi_j(W_j,z)
    =t_j F_j'(W_j),
\]
and therefore
\[
 \mathfrak g_j(z)=\frac{1}{\sqrt{2\pi t_j F_j'(W_j(z))}},
\]
with sign fixed by the oriented contour. Differentiating $F_j(W_j(z))=z$
gives the equivalent squared identity
\[
 2\pi t_j\mathfrak g_j(z)^2=W_j'(z).
\]
Its right-hand side converges uniformly on the annulus to $W'(z)$,
which extends holomorphically and without zeros across $\mathcal D$.
Rouch\'e's theorem therefore gives winding number zero for
$\mathfrak g_j^2$, hence also for $\mathfrak g_j$.
This conclusion uses extension of the limiting squared amplitude across
the disk; pointwise nonvanishing on the annulus would not suffice.

The curve $W_j(\partial\mathcal D)$ winds once about every input root in $\mathcal D_b$
and zero times about every other input root. These winding numbers are
stable because the limiting curve avoids $B$ and $W$ is one-to-one.
Also $W_j(z)-z$ has zero winding on $\partial\mathcal D$, since its limit
$W(z)-z$ is holomorphic and nonzero on $\mathcal D$. The leading expression in
\eqref{eq:outlier-annulus-asymptotic} thus has winding number exactly
$(\divz H_j)(\mathcal D_b)$. The relative error is holomorphic on the annulus and eventually has
modulus less than one, so its addition does not change the winding
number. The argument principle proves
\eqref{eq:outlier-count}. Since the input cluster is nonempty, every
such disk contains a descendant zero. Together with weak convergence,
this proves the lower estimate in \eqref{eq:finite-cluster-fixed-set}.

\smallskip\noindent\emph{3. Pooling and the endpoint ratios.}
For the lower estimate in \eqref{eq:finite-cluster-pooled}, weak
convergence supplies $S_\eta$, order zero supplies $B$, and
\eqref{eq:outlier-count} supplies each $O_t(\eta,B)$ at an interior
ratio. A finite covering of their compact closure gives the required
Hausdorff estimate.

For the upper estimate, take a convergent sequence of derivative zeros
and pass to $m_j/D_j\to t\in[0,1]$. Interior ratios are already
settled. At $t=0$, on compact sets avoiding $B\cup Z(C_\eta)$ use
\[
 \mathscr A_{j,z}(u)=
 \exp\!\left(\frac1{D_j}\log\frac{H_j(z+u)}{H_j(z)}\right).
\]
Its constant term is one, its linear coefficient tends uniformly to
$C_\eta(z)\ne0$, and the coefficient of its $D_j$th power at order
$m_j$ is $H_j^{(m_j)}(z)/(m_j!H_j(z))$.
Lemma~\ref{lem:small-order-coefficients} excludes zeros for
$m_j=o(D_j)$; order zero is covered by the root-set assumption. Lower-endpoint limits
therefore lie in $B\cup Z(C_\eta)\subset B\cup S_\eta$.

At $t=1$, list the roots with multiplicity as
$a_{1,j},\ldots,a_{D_j,j}$ and apply the same lemma to
\[
 \mathscr B_{j,z}(u)=
 \exp\!\left(\frac1{D_j}\sum_{\ell=1}^{D_j}
                              \log(1+(z-a_{\ell,j})u)\right).
\]
Its linear coefficient tends to $z-\mathfrak c_\eta$, and its $D_j$th power is
$u^{D_j}H_j(z+u^{-1})$. The coefficient at order
$D_j-m_j=o(D_j)$ is $H_j^{(m_j)}(z)/m_j!$. Every upper-endpoint limit
is therefore $\mathfrak c_\eta\in S_\eta$, completing the proof.
\end{proof}

\subsection{Endpoint universality and an integer conservation law}
The small-order argument also applies to general root arrays. It
requires the root locations as well as the limiting measure, but no
rate of clustering and no nondegeneracy condition.

\begin{theorem}[Sublinear differentiation of arbitrary finite clusters]
\label{thm:cluster-endpoint-counts}
Let $H_j$ have degree $D_j\to\infty$, suppose
$D_j^{-1}\divz H_j\Longrightarrow\eta$, and suppose
$Z(H_j)\to B$ in Hausdorff distance, where $B$ is finite and
$\eta=\sum_{a\in A}p_a\delta_a$ has positive masses and $A\subseteq B$.
Set $E=B\cup Z(C_\eta)$. For $x\in E$, let
$\iota_x=\operatorname{ord}_x C_\eta$: this is $-1$ at a pole,
its positive zero multiplicity at a zero, and $0$ elsewhere.
Choose disjoint sufficiently small open disks $\mathcal D_x$ about the points
of $E$. For any integers $M_j\ge0$ with $M_j/D_j\to0$, eventually
for every $0\le m\le M_j$,
\begin{equation}\label{eq:cluster-endpoint-counts}
 (\divz H_j^{(m)})(\mathcal D_x)
   =(\divz H_j)(\mathcal D_x)+m\iota_x\qquad(x\in E).
\end{equation}
All zeros of $H_j^{(m)}$ lie in these disks. Moreover, for
$1\le m_j<D_j$,
\begin{align}
 m_j/D_j\to0&\quad\Longrightarrow\quad
 Z(H_j^{(m_j)})\longrightarrow B\cup Z(C_\eta),
 \label{eq:cluster-low-endpoint}\\
 m_j/D_j\to1&\quad\Longrightarrow\quad
 Z(H_j^{(m_j)})\longrightarrow\{\mathfrak c_\eta\}.
 \label{eq:cluster-high-endpoint}
\end{align}
Both limits are in Hausdorff distance.
\end{theorem}
\begin{proof}
Apply the proof of \cite[Theorem~4.3]{HS1} with the following array
substitution. Set $C_j=H_j'/(D_jH_j)$. By
\eqref{eq:cluster-log-convergence}, $C_j\to C_\eta$ locally uniformly,
with derivatives, away from $B$. Near the disk boundaries the functions
\[
 \mathscr A_{j,z}(u)=\exp\left(\int_0^u C_j(z+v)\dd v\right),
 \qquad
 \mathscr A_{j,z}(u)^{D_j}=\frac{H_j(z+u)}{H_j(z)}
\]
satisfy the varying-family hypotheses of
Lemma~\ref{lem:small-order-coefficients}. The winding-number and
telescoping argument of the cited proof therefore applies with exponent
$D_j$, limiting logarithmic derivative $C_\eta$, and initial counts
$(\divz H_j)(\mathcal D_x)$. It gives
\eqref{eq:cluster-endpoint-counts} simultaneously for $0\le m\le M_j$;
since $\sum_{x\in E}\iota_x=-1$, these counts exhaust the degree.

For the lower Hausdorff limit, the same proof applies at the atoms and
at the zeros of $C_\eta$. At each additional location
$x\in B\setminus A$, the initial count is eventually at least one and
$\iota_x\ge0$, so that cluster also remains nonempty. Shrinking the
disks gives \eqref{eq:cluster-low-endpoint}. The upper-endpoint limit
is the reverse-coefficient argument of \cite[Theorem~4.2]{HS1},
applied to the normalized varying families already used at the end of
the proof of Theorem~\ref{thm:finite-cluster-universality}; their
linear coefficients tend to $z-\mathfrak c_\eta$.
\end{proof}

In Figure~\ref{fig:trails}(b), where $\eta=\eta_0$ and
$C_{\eta_0}(w)=w/(w^2-1)$, the rows of small ratio emerge from $0$, the
zero of $C_{\eta_0}$; near $0$ each derivative order adds one zero,
while near each atom $\pm1$ it removes one, and the simple roots $2+i$
and $3+i$ keep one zero each.

\subsection{Consequences and examples}
\begin{corollary}[Clusters with no additional limiting locations]
\label{thm:cluster-universality}
Let $P$ be monic of degree $d\ge2$ and $\deg H_j=D_j\to\infty$.
Suppose
\begin{equation}\label{eq:cluster-input}
 \frac1{D_j}\divz H_j\Longrightarrow\rootmu_P,
 \qquad
 \max_{a\in Z(H_j)}\dist(a,Z(P))\longrightarrow0.
\end{equation}
If $m_j/D_j\to t\in(0,1)$, then
\begin{equation}\label{eq:cluster-fixed}
 \frac1{D_j-m_j}\divz H_j^{(m_j)}\Longrightarrow\mu_{t,\rootmu_P},
 \qquad Z(H_j^{(m_j)})\longrightarrow\FixSh_{t,P}.
\end{equation}
Moreover,
\begin{equation}\label{eq:cluster-pooled}
 \pi_{H_j}\Longrightarrow\rho_P,
 \qquad\mathcal Z(H_j)\longrightarrow\Sh_P.
\end{equation}
The set limits are Hausdorff limits. No convergence rate, within-cluster
separation, equality of multiplicities, or simplicity assumption is
required.
\end{corollary}
\begin{proof}
Every root of $P$ has positive limiting mass, so
\eqref{eq:cluster-input} implies $Z(H_j)\to Z(P)$.
Apply Theorem~\ref{thm:finite-cluster-universality} with $B=A=Z(P)$.
There are no additional transport paths.
\end{proof}

\begin{corollary}[Limits of the input shadows]
\label{cor:cluster-entire-shadows}
Under \eqref{eq:finite-location-input},
\begin{equation}\label{eq:entire-finite-cluster}
 \Sh_{H_j}\longrightarrow K_{\eta,B},\qquad
 \AllSh_{H_j}\longrightarrow K_{\eta,B}
 \quad\text{in Hausdorff distance}.
\end{equation}
Under \eqref{eq:cluster-input}, both limits are $\Sh_P$, not in
general $\AllSh_P$.
\end{corollary}
\begin{proof}
For each fixed $H_j$, the core convergence in
Proposition~\ref{prop:shadow-foundations} supplies $k_j\ge1$ with
$d_{\mathrm H}(\mathcal Z_{H_j,k_j},\Sh_{H_j})<1/j$.
The polynomial $H_j^{k_j}$ has the same normalized root measure and
root locations as $H_j$, and its degree tends to infinity. Hence
Theorem~\ref{thm:finite-cluster-universality} gives
$\mathcal Z_{H_j,k_j}=\mathcal Z(H_j^{k_j})\to K_{\eta,B}$.

For the upper full-shadow estimate, approximate any proposed offending
point of $\AllSh_{H_j}$ within $1/j$ by a zero of a derivative of
$H_j^{\ell_j}$, for some $\ell_j\ge1$. These polynomials satisfy the
same limiting data, so their derivative zeros approach $K_{\eta,B}$.
The common root hull gives compactness. The lower estimate follows
from $\mathcal Z(H_j)\subseteq\AllSh_{H_j}$ and
\eqref{eq:finite-cluster-pooled}.
\end{proof}

\begin{corollary}[Continuity with mass and location data]
\label{cor:mass-location-continuity}
Let $\eta_j$ be finitely supported probability measures and $B_j$
finite sets containing their supports. If
$\eta_j\Longrightarrow\eta$ and $B_j\to B$ in Hausdorff distance,
where $B$ is finite, then
\[
 K_{\eta_j,B_j}\longrightarrow K_{\eta,B}.
\]
If also $t_j\to t\in(0,1)$, then
\[
 S_{t_j,\eta_j}\cup O_{t_j}(\eta_j,B_j)
   \longrightarrow S_{t,\eta}\cup O_t(\eta,B)
\]
in Hausdorff distance. Thus complete limits remain continuous when
positive masses vanish or clusters merge, provided the limiting
locations are retained.
\end{corollary}
\begin{proof}
For each $j$, approximate $\eta_j$ by normalized integer multiplicities
at the points of $B_j$, assigning at least one root to each point,
including those of zero $\eta_j$-mass.
Theorem~\ref{thm:finite-cluster-universality} applies as the degree
increases. Choose an approximating polynomial $G_j$ of degree
$L_j\to\infty$ whose root measure is within $1/j$ of $\eta_j$ in a
weak metric and whose pooled zero set is within $1/j$ of
$K_{\eta_j,B_j}$. For the second assertion, also require that the zero
set of its derivative of order $k_j=\lfloor t_jL_j\rfloor$ be within
$1/j$ of $S_{t_j,\eta_j}\cup O_{t_j}(\eta_j,B_j)$.

The root sets of $G_j$ converge to $B$ and their root measures to
$\eta$. Apply the same theorem to this diagonal sequence. All measures
and sets have a common compact root hull, and the two conclusions
follow. For the first assertion alone, only the pooled approximation is
needed.
\end{proof}

\begin{example}[One positive mass and finitely many extra locations]
\label{ex:vanishing-root}
For $\eta=\delta_a$, direct calculation gives
\[
 \mu_{t,\eta}=\delta_a,\qquad S_\eta=\{a\},\qquad
 W_{t,\eta}(z)=\frac{z-ta}{1-t},\qquad
 \theta_b(t)=(1-t)b+ta.
\]
Every $b\ne a$ is selected at every interior ratio. Thus for a finite
set $B$ containing $a$,
\begin{equation}\label{eq:one-mass-star}
 K_{\delta_a,B}=\bigcup_{b\in B}[a,b],
\end{equation}
where $[a,b]=\{(1-u)a+ub:0\le u\le1\}$. The additional star has zero
limiting pooled mass.

The introductory example $H_D(z)=z^{D-1}(z-1)$ is the case
$a=0$, $B=\{0,1\}$.
\end{example}

\begin{corollary}[Fixed polynomial factors]\label{cor:fixed-factor}
Let $P,Q$ be nonzero polynomials, with $\deg P=d\ge1$, and put
$H_n=P^nQ$. After monic normalization, all conclusions of
Theorem~\ref{thm:finite-cluster-universality} and
Corollary~\ref{cor:cluster-entire-shadows} hold with
\[
 \eta=\rootmu_P,\qquad B=Z(P)\cup Z(Q).
\]
For $b\in Z(Q)\setminus Z(P)$ with $P'(b)\ne0$, the transport is
\[
 \theta_b(t)=b-dt\frac{P(b)}{P'(b)}.
\]
Its selected zeros retain the multiplicity of $b$ in $Q$ while they
are separated from the fixed-ratio support. If
$Z(Q)\subseteq Z(P)\cup Z(P')$, the pooled set limit is $\Sh_P$.
\end{corollary}
\begin{proof}
The normalized root measures of $P^nQ$ tend to $\rootmu_P$ and their root sets
equal $B$, so the stated results apply. For the final assertion,
each additional location has $C_{\rootmu_P}(b)=0$ and hence no
transport path; it already lies in $\Sh_P$ by
\eqref{eq:weighted-endpoints} and
\eqref{eq:polynomial-weighted-identification}.
\end{proof}

The weak limit depends only on $\eta$. The complete set limit depends
on $(\eta,B)$, through the selected paths in
\eqref{eq:complete-weighted-shadow}.
\section{Definability and finite topology}
\label{sec:tame}

The complete cluster limits describe the limiting zero sets. To build
loops in the core of a fixed polynomial $P$, we also need continuous
paths through its \emph{support family}
$\{\FixSh_{t,P}\}_{0<t<1}$ of fixed-ratio supports. We obtain them from
definability in the o-minimal structure $\Rexp$, and at the same time
bound the topology and boundary length in terms of the number of root
locations.

\subsection{The real exponential structure}
Write
$\Rexp=(\R;<,+,\cdot,\exp)$.
A set is \emph{definable} in $\Rexp$ if it is described by a finite
formula using real constants, order, the field operations, the
exponential function, and quantification over real variables. A function
is definable when its graph is. We allow real parameters throughout.
An expansion of the
ordered real field is \emph{o-minimal} if every definable subset of
$\R$ is a finite union of points and intervals. The real exponential
field is o-minimal, and we use the following consequences
\cite{Wilkie,vdD}.

Definable sets are closed under projection, complement, and closure. They are also closed
under \emph{fiberwise closure}: for a family with parameters, one may take the closure
separately in each parameter fiber. Connected components in a definable family form a
definable family, and their number is uniformly bounded.

We recall the geometric terminology used below. The \emph{rank} of a
smooth map on a manifold is the rank of its real differential on the
tangent space. A finite $C^k$ \emph{stratification} of a set $E$ is a
finite partition into disjoint $C^k$ submanifolds, called strata, with
the frontier condition: the closure of each stratum \emph{relative to
$E$} is a union of strata, and its frontier consists of strata of
smaller dimension. A \emph{cell decomposition} is a finite compatible
partition into definable pieces homeomorphic to open boxes of their
respective dimensions, with a point as a zero-dimensional cell.
Compatibility means that each of the specified sets is a union of
pieces. The dimension of a definable set is the largest dimension of
one of its strata.

Definable maps admit stratifications of constant rank. A definable
family also admits a finite \emph{topological trivialization}: its
parameter space can be partitioned into finitely many definable sets
over each of which the fibers are identified by definable
homeomorphisms preserving the parameter. Finally, a compact definable
set is \emph{triangulable}, meaning that it is homeomorphic to a finite
simplicial complex. A finite simplicial complex is a finite collection
of Euclidean simplices closed under taking faces, with intersections
that are common faces. These results are used only in their
qualitative form.

\begin{theorem}[Definability of the barrier and the support families]
\label{thm:definable}
For monic degree-$d$ polynomials, the normalized barrier
$\Blev_{\rootmu_P}(z,t)$ is $\Rexp$-definable on
$P(z)\ne0$, $0<t<1$. The families
\[
 \{(P,t,z):0<t<1,\ z\in\FixSh_{t,P}\},
 \qquad \{(P,z):z\in\Sh_P\}
\]
are definable as well. The same assertions hold for the weighted
barrier $\Blev_\eta$ and supports $S_{t,\eta}$ and $S_\eta$ when the
number of atoms is fixed and their distinct locations and positive
masses vary.
\end{theorem}
\begin{proof}
The logarithm is definable on $(0,\infty)$ as the inverse of the
exponential. Thus the weighted landscape and its superlevel sets in
\eqref{eq:weighted-barrier-definition} form definable families. The
two positive ends are marked points in sphere coordinates. The
condition that they belong to the same fiberwise connected component
is definable. Taking the supremum of the levels for which this
condition holds gives $\Blev_\eta$.

The supports $S_{t,\eta}$ and $S_\eta$ can be recovered without
integration in $t$. For a
definable function, the existence and continuity of partial derivatives
are definable conditions: both can be expressed by real-variable
limit formulas. It is therefore definable that
$\Blev_\eta(\cdot,t)$ be $C^2$ with zero Laplacian on a given root-free
open disk. By \eqref{eq:weighted-potential}, this is precisely the
condition that the disk be disjoint from $S_{t,\eta}$. Taking the
complement of the union of such disks gives the support off the atoms.
Take its fiberwise closure and add every atom with $p_a>t$, as
prescribed by \eqref{eq:weighted-atoms}. An atom with no atomic mass
can belong to the support only as a limit of other support points, so
this operation gives the entire $S_{t,\eta}$. Projection over
$0<t<1$ and fiberwise closure give $S_\eta$.

For polynomial coefficients, one can either use the normalized
landscape $d^{-1}\log|P(w)|-t\log|w-z|$ directly or pass to the
finitely many multiplicity patterns. Root locations and their
multiplicities are definable from $P$ and its first $d$ derivatives.
The condition $p_a>t$ becomes $\operatorname{ord}_a(P)>dt$.
The identities $\FixSh_{t,P}=S_{t,\rootmu_P}$ and
$\Sh_P=S_{\rootmu_P}$ complete the proof.
\end{proof}

\begin{theorem}[Complexity controlled by the number of distinct roots]
\label{thm:atomic-complexity}
For each integer $s\ge1$, the family of limiting cores $S_\eta$
for probability measures with at most $s$ atoms is $\Rexp$-definable
in the atom locations and positive masses. Its members have only finitely
many homeomorphism types. In particular their Betti numbers (the dimensions of
their rational homology groups) and numbers of complementary components are bounded in terms of $s$ alone.

There is an integer $M_s\ge1$ such that each boundary $\partial S_\eta$
is the union of at most $M_s$ points and at most $M_s$ bounded $C^1$
arcs on each of which both Cartesian coordinates are monotone.
Consequently, if $\conv(\supp\eta)$ lies in a disk of radius $R$, then
\begin{equation}\label{eq:root-number-length}
 \Haus^1(\partial S_\eta)\le4M_sR.
\end{equation}
For a polynomial with at most $s$ distinct roots the same bound holds
for $\partial\AllSh_P$. Thus these topology and length bounds are
independent of the degree and of the root multiplicities. The topology
bounds concern the limiting core; only the length assertion extends
here to the full-shadow boundary.
\end{theorem}
\begin{proof}
For exactly $r\le s$ atoms, use the real and imaginary parts of
their locations and the positive parameters $p_1,\ldots,p_r$ with sum
one. Theorem~\ref{thm:definable} gives a definable family of cores.
Taking the finite union over $r\le s$ and applying definable
topological triviality proves the assertions about homeomorphism
types and Betti numbers. Applied to the family of complements,
component finiteness gives the uniform complementary-component bound.

To obtain a uniform length bound, apply cell decomposition to the
entire boundary family, compatibly with projection to parameter space. Its
one-dimensional fibers decompose into a uniformly bounded number of
vertical arcs and graphs $y=f(x)$ on which $f$ is $C^1$.
Refine the graph pieces by the sign of $\partial_xf$ and use definable
monotonicity. Their number, and the number of point fibers left over,
is bounded by one integer $M_s$, uniformly over all atom parameters.
Both coordinates are monotone on each resulting arc. In a disk of
radius $R$, its length is at most the sum of the two coordinate
variations, hence at most $4R$. Summing proves
\eqref{eq:root-number-length}. Arc endpoints can be adjoined without
changing the length. Finally,
Proposition~\ref{prop:shadow-foundations} shows that the full shadow adds
only the countable transient set $T_P$ to the core. Adding a countable
set does not change one-dimensional Hausdorff measure.
\end{proof}

\begin{remark}
The bounds depend on the number of distinct atoms, not on polynomial
degree or multiplicity. They concern topology and length: they do not
assert algebraicity of the boundary for irrational masses, nor an
algebraic-degree bound depending only on $s$.
\end{remark}

\begin{corollary}[Finite topology]\label{cor:tame}
In each fixed degree, limiting cores have only finitely many homeomorphism types. Their
Betti numbers are bounded by a degree-dependent constant. Each $\Sh_P$ is homeomorphic to a finite simplicial complex of
dimension at most two, has finitely many connected components and finitely many
complementary components, and satisfies
\[
 \Haus^1(\partial\Sh_P)<\infty.
\]
For fixed $P$, the topology of $\FixSh_{t,P}$ changes at only finitely many ratios
$t\in(0,1)$.
\end{corollary}
\begin{proof}
Apply definable topological triviality to the two families in Theorem~\ref{thm:definable},
and compact triangulation to each limiting core. The boundary of a definable planar set
has dimension at most one, so it is a finite union of points and bounded definable $C^1$
arcs. Subdivide each parametrized arc until both coordinate functions are monotone. Its
length is then at most the sum of the variations of those coordinates and is therefore
finite. The assertion about the ratio parameter follows from trivialization over the
one-dimensional interval $(0,1)$.
\end{proof}

\subsection{The compact incidence space}
For fixed $P$, define its \emph{support incidence space} by
\begin{equation}\label{eq:incidence}
 \Inc_P=\overline{\{(t,z):0<t<1,\ z\in\FixSh_{t,P}\}}
 \subset[0,1]\times\conv Z(P),\qquad \prz(t,z)=z.
\end{equation}
A \emph{fiber} is the set of spatial points with a prescribed value of
$t$. The endpoint fibers are
$Z(P)\cup Z(C_{\rootmu_P})$ at $t=0$ and
$\{\mathfrak c_P\}$ at $t=1$, by
\eqref{eq:weighted-endpoints}. For $\deg P\ge2$, the lower fiber is
also $Z(P)\cup Z(P')$. The incidence space is compact and definable,
its fibers are Hausdorff-continuous on $[0,1]$, and
$\prz(\Inc_P)=\Sh_P$. Figure~\ref{fig:trails}(b) shows the analogous
lift of a complete zero set, with the zeros of $P_N^{(m)}$ placed at
height $t=m/D$.

\subsection{Transport of vanishing-mass clusters}

\begin{proposition}[Only finitely many straight transport segments]
\label{prop:transport-segments}
For the data $(\eta,B)$ in \eqref{eq:finite-location-input}, the complete limiting set
$K_{\eta,B}$ is the union of $S_\eta$ and finitely many closed straight segments. Each added
segment has at least one endpoint in $S_\eta$. In particular,
\begin{equation}\label{eq:transport-body}
 \operatorname{int}K_{\eta,B}=\operatorname{int}S_\eta,
 \qquad \Area(K_{\eta,B})=\Area(S_\eta).
\end{equation}
A directly specified enclosure, requiring no saddle selection, is
\begin{equation}\label{eq:transport-enclosure}
 K_{\eta,B}\subseteq S_\eta\cup B\cup
 \bigcup_{\substack{b\in B\setminus A\\ C_\eta(b)\ne0}}
       [b,b-C_\eta(b)^{-1}].
\end{equation}
The statement allows arbitrary real weights and requires no
algebraicity assumption on $S_\eta$.
\end{proposition}
\begin{proof}
Theorem~\ref{thm:definable} gives definability of the weighted
barrier and the supports $S_{t,\eta}$ and $S_\eta$. The
selected inverse $W_{t,\eta}$ is also definable. Away from $A$, it is
obtained by spatial differentiation of \eqref{eq:weighted-potential};
at points of $A$ outside the support, use its continuous extension.

For $b\in B\setminus A$ with $C_\eta(b)\ne0$, set
\[
 I_b:=\{t\in(0,1):\theta_b(t)\notin S_\eta,
                    \ W_{t,\eta}(\theta_b(t))=b\}.
\]
This is definable. It is open: at a selected point the inverse is simple, and the implicit
function theorem identifies its nearby continuation with the fixed solution $w=b$ of
$F_{t,\eta}(w)=\theta_b(t)$. O-minimality implies that $I_b$ is a finite union of open
intervals. On any component of
$\{t:\theta_b(t)\notin S_\eta\}$, the same equality is both open and closed, by local
persistence and continuity of the selected inverse. Thus every endpoint of an interval of
$I_b$ lying strictly between zero and one is mapped into $S_\eta$.

The same conclusion holds for a right endpoint equal to one. Otherwise its image has a
neighborhood disjoint from $S_\eta$, whereas \eqref{eq:weighted-endpoints} and
\eqref{eq:weighted-W} give $W_{t,\eta}(z)\to\infty$ uniformly on compact subsets of that
neighborhood as $t\uparrow1$, contradicting $W_{t,\eta}(\theta_b(t))=b$. If $b\notin
S_\eta$, then $C_\eta(b)\ne0$ by \eqref{eq:weighted-endpoints}. As $t\downarrow0$,
$W_{t,\eta}(z)\to z$ locally off $S_\eta$ and $F_{t,\eta}\to\operatorname{id}$ near $b$.
Uniqueness of the nearby inverse implies $(0,\varepsilon)\subset I_b$ for some $\varepsilon>0$.
Consequently every $b\in B\setminus S_\eta$ equals $\theta_b(0)$, the
image of the left endpoint $0$ of one of these interval closures.

It follows directly from \eqref{eq:complete-weighted-shadow} that
\[
 K_{\eta,B}=S_\eta\cup
       \bigcup_{\substack{b\in B\setminus A\\C_\eta(b)\ne0}}
       \theta_b(\overline{I_b}).
\]
Each interval closure has a right endpoint mapped into $S_\eta$, and its image is a closed
straight segment. This proves the finite-segment assertion and
\eqref{eq:transport-enclosure}. A finite union of segments has zero planar area and cannot
create an interior point outside the closed set $S_\eta$, giving \eqref{eq:transport-body}.
\end{proof}

In Figure~\ref{fig:cubic} the four selected segments end at points of
$S_\eta$, marked by circles.

\subsection{Continuous paths through the support family}

Loops in the core are assembled from paths $t\mapsto\gamma(t)\in\FixSh_{t,P}$
along increasing ratios. We use definable choice, finite discontinuities
of definable one-variable maps, bounded definable endpoint limits, and
definable cell decomposition \cite{vdD}.

\begin{lemma}[Continuous selection on an interval]
\label{lem:interval-selection}
Let $E\subset [t_0,t_1]\times\R^r$ be a compact definable set, and write
$E_t=\{x:(t,x)\in E\}$. Suppose its fibers are nonempty and
Hausdorff-continuous in $t$. Given $x_0\in E_{t_0}$, there is a continuous
definable function $\gamma:[t_0,t_1]\to\R^r$ such that
$\gamma(t_0)=x_0$ and $\gamma(t)\in E_t$ for every $t$.
\end{lemma}
\begin{proof}
If $t_0=t_1$, the constant section suffices. Otherwise choose a finite
definable cell decomposition of $E$ whose box charts make the time
coordinate either constant or one of the box coordinates. Together
with $t_0$ and $t_1$, the frontiers of the cell time images partition
$[t_0,t_1]$ into finitely many points and open intervals. Any cell meeting
one of these open intervals projects onto the entire interval.

Suppose a value $x_c\in E_c$ has been prescribed at the left endpoint
$c$ of one such interval $(c,e)$. Hausdorff continuity implies that
$(c,x_c)$ belongs to the closure of $E\cap((c,e)\times\R^r)$.
On $[c,e]$, compactness and definable choice give a definable map
$\gamma(t)\in E_t$ minimizing $\|\gamma(t)-x_c\|_\infty$.
In particular $\gamma(c)=x_c$, and Hausdorff continuity gives
$\gamma(t)\to x_c$ as $t\downarrow c$.
A definable curve has only finitely many discontinuities. Shortening
the interval therefore supplies a continuous definable section on
$[c,c+\delta]$ with the prescribed value at $c$.

The definable sets of times at which this local section lies in
each cell have finite frontiers. Shortening $\delta$ again, the
section lies in one cell for $c<t<c+\delta$.
Choose $c_1\in(c,c+\delta)$. In this cell's box chart, keep every
coordinate except time fixed at the inverse image of $\gamma(c_1)$.
For $t\ge c_1$ this continues the section across $(c,e)$, since
the cell projects onto that entire interval. Its spatial coordinates are bounded and
definable, so each has a limit as $t\uparrow e$.
Compactness places that limit in $E_e$. Use it as the
prescribed value for the next interval. Finitely many repetitions
complete the section. This proof does not assume that the fibers are
connected or convex.
\end{proof}

\begin{lemma}[Nested convex hulls]
\label{lem:nested-support-hulls}
For every finite positive atomic probability measure $\eta$ and
$0<t<t'<1$,
\[
 \conv S_{t',\eta}\subseteq\conv S_{t,\eta}.
\]
The centroid $\mathfrak c_\eta$ belongs to every one of these convex hulls.
For each fixed $\eta$ and $t_0\in(0,1)$, the family obtained by adjoining
$S_{1,\eta}=\{\mathfrak c_\eta\}$ is compact, definable, and
Hausdorff-continuous on $[t_0,1]$. In particular, every point of
$S_{t_0,\eta}$ can be joined to $\mathfrak c_\eta$ by a continuous path
$\gamma(t)\in S_{t,\eta}$, $t_0\le t\le1$.
\end{lemma}
\begin{proof}
Approximate $\eta$ by positive integer multiplicities at its atom
locations, with total degree $D_n\to\infty$. Apply Gauss--Lucas to
the derivatives at orders $\lfloor tD_n\rfloor$ and
$\lfloor t'D_n\rfloor$, and pass to their fixed-ratio Hausdorff limits.
Convex hull is nonexpansive in Hausdorff distance, giving the stated
inclusion for arbitrary positive real weights. Letting $t'\uparrow1$
and using the upper-endpoint limit places the centroid in every hull. Definability and continuity of the family
follow from Theorem~\ref{thm:definable} and
Proposition~\ref{prop:weighted-foundations}, including the endpoint
$t=1$. Apply Lemma~\ref{lem:interval-selection} for the path assertion.
\end{proof}

\begin{corollary}[Paths following increasing derivative ratios]
\label{cor:ratio-paths}
For every nonconstant polynomial $P$ and every $z\in\Sh_P$, there
is a $t_z\in[0,1]$ and a continuous definable path
$\gamma_z:[t_z,1]\to\Sh_P$ such that
\[
 \gamma_z(t_z)=z,\qquad \gamma_z(1)=\mathfrak c_P,\qquad
 \gamma_z(t)\in\FixSh_{t,P}\quad(t_z<t<1).
\]
The choices are not asserted to depend continuously on $z$.
\end{corollary}
\begin{proof}
For $\deg P\ge2$, adjoin to the fixed-ratio family the endpoint
fibers $E_P(0)=Z(P)\cup Z(P')$ and $E_P(1)=\{\mathfrak c_P\}$,
and set $E_P(t)=\FixSh_{t,P}$ for $0<t<1$; these are the fibers of
$\Inc_P$. The endpoint limits \eqref{eq:weighted-endpoints}, the Hausdorff
continuity in Proposition~\ref{prop:weighted-foundations}, and
definability make this a compact Hausdorff-continuous definable family
on $[0,1]$. Its union is compact and, by \eqref{eq:weighted-shadows},
\eqref{eq:polynomial-weighted-identification}, and the endpoint limits,
equals $\Sh_P$. Hence $z\in E_P(t_z)$
for some $t_z$. Apply Lemma~\ref{lem:interval-selection} on $[t_z,1]$
with prescribed initial value $z$. The degree-one case is immediate.
\end{proof}

These paths lie in the core, not in a finite-level approximation.
They do not yield a contraction: that would require a continuous
choice jointly in the initial point and the path parameter, and the existence of individual paths provides no such joint choice.

\subsection{Connectedness with mass and location data}
\begin{corollary}[Connectedness for finite mass and location data]
\label{cor:weighted-connected}
The sets $S_\eta$ and $K_{\eta,B}$ are connected for every finite positive-mass measure
$\eta$ and every finite $B\supseteq\supp\eta$.
\end{corollary}
\begin{proof}
Approximate $\eta$ by normalized positive integer multiplicities at
all points of $B$, retaining at least one root at every zero-mass
location. Their polynomial cores are connected by
Proposition~\ref{prop:shadow-foundations} and converge to
$K_{\eta,B}$ by Corollary~\ref{cor:cluster-entire-shadows}.
A compact Hausdorff limit of nonempty connected sets is connected.
Taking $B=\supp\eta$ gives $S_\eta$.
\end{proof}
\section{Holes and the failure of contractibility}
\label{sec:holes}

We now turn the selected transport segments into holes of fixed
polynomial cores. The construction has two independent parts: paths in
the support family produce loops, and complete zero-set convergence
keeps specified points outside the full shadow. Both are needed; a
cycle in an auxiliary Hausdorff limit would not suffice.

A \emph{homotopy} of paths is a continuous deformation with their
endpoints fixed. For a continuous closed curve avoiding a point,
its winding number is $1/(2\pi)$ times the change in a continuous
argument along the curve. This extends the integral convention for
piecewise smooth curves; it is integer-valued and unchanged by a
homotopy avoiding the point.

A \emph{finite embedded graph} consists of finitely many vertices and
simple arcs that meet only at their endpoints. A \emph{cycle} is a
closed edge path with no repeated vertex except its first and last.
A \emph{tree} is a connected graph with no cycle; a \emph{spanning
tree} is a tree subgraph containing every vertex. Crossings are treated
as vertices by subdividing the arcs.

\subsection{Transport over two real roots}
The base measure is
\[
 \eta_0=\tfrac12(\delta_{-1}+\delta_1),\qquad
 C_{\eta_0}(w)=\frac{w}{w^2-1}.
\]
Its fixed-ratio supports lie in $[-1,1]$, by Gauss--Lucas and
polynomial approximation. Its core contains both
endpoints by \eqref{eq:weighted-atoms}. By the polynomial identification
in Proposition~\ref{prop:weighted-foundations}, it is the core of $z^2-1$,
which is connected by Proposition~\ref{prop:shadow-foundations}; hence
$S_{\eta_0}=[-1,1]$.

\begin{lemma}[Selected transport for the quadratic base]
\label{lem:quadratic-transports}
For $b=x+iy$ with $x,y>0$, set
\begin{equation}\label{eq:quadratic-transport-data}
 \tau_b=\frac{|b|^2}{1+|b|^2},\qquad
 \beta(b)=\frac{2x}{1+|b|^2},\qquad
 \theta_b(t)=(1-t)b+\frac{t}{b}.
\end{equation}
For $0<t<\tau_b$, the point $\theta_b(t)$ is in the upper half-plane,
outside $S_{t,\eta_0}$, and
$W_{t,\eta_0}(\theta_b(t))=b$. It lands at
$\beta(b)\in S_{\tau_b,\eta_0}$ when $t=\tau_b$, and $b$ is not
selected at its transported point for $t>\tau_b$. Thus, for any finite
set $B$ consisting of $-1,1$ and distinct points in the open first
quadrant,
\begin{equation}\label{eq:quadratic-transport-graph}
 K_{\eta_0,B}=[-1,1]\cup
     \bigcup_{b\in B\setminus\{-1,1\}}[\beta(b),b].
\end{equation}
\end{lemma}
\begin{proof}
The transport map is
\[
 F_{t,\eta_0}(w)=(1-t)w+\frac{t}{w}.
\]
Since the fixed-ratio measure is supported on the real line, its
Cauchy transform has negative imaginary part in the upper half-plane.
Formula~\eqref{eq:weighted-W} therefore shows that the selected inverse
maps the upper half-plane into itself. Now
\[
 \Imm\theta_b(t)=y\left(1-t-\frac{t}{|b|^2}\right),
\]
which is positive precisely for $t<\tau_b$. The two solutions of
$F_{t,\eta_0}(w)=\theta_b(t)$ are $b$ and $t/((1-t)b)$; the second
is in the lower half-plane. The selected inverse is consequently $b$.
For $t>\tau_b$, the observation point and its selected inverse are
both in the lower half-plane, so $b$ is no longer selected.

At $t=\tau_b$, the observation point is $\beta(b)$. Were it outside
the real support, its Cauchy transform, and hence its selected inverse,
would be real. Continuity from smaller ratios would instead give the
nonreal value $b$. This contradiction proves the landing assertion.
The complete-set formula~\eqref{eq:complete-weighted-shadow} now gives
\eqref{eq:quadratic-transport-graph}.
\end{proof}

For $x>\sqrt2$, set
\begin{equation}\label{eq:hole-spike-data}
 b_x=x+i,\qquad
 \tau_x=\frac{x^2+1}{x^2+2},\qquad
 \beta_x=\frac{2x}{x^2+2},\qquad
 \theta_x(t)=x-\frac{x^3}{x^2+1}t
             +i\left(1-\frac{x^2+2}{x^2+1}t\right).
\end{equation}
Thus $\theta_x(0)=b_x$ and $\theta_x(\tau_x)=\beta_x\in(0,1)$.
Write $\Gamma_x=[\beta_x,b_x]$ for this straight segment.

\begin{lemma}[The complete fixed-ratio attachment]
\label{lem:quadratic-spikes}
Let $x_1<\cdots<x_k$ be distinct numbers greater than $\sqrt2$, and put
$B=\{-1,1,b_{x_1},\ldots,b_{x_k}\}$. Then, for $0<t<1$,
\begin{equation}\label{eq:hole-limit-fibers}
 S_{t,\eta_0}\cup O_t(\eta_0,B)
 =S_{t,\eta_0}\cup\{\theta_{x_j}(t):t<\tau_{x_j},\ 1\le j\le k\}.
\end{equation}
At $t=\tau_{x_j}$ the point $\beta_{x_j}$ belongs to $S_{t,\eta_0}$.
Consequently
\begin{equation}\label{eq:hole-limit-graph}
 K_{\eta_0,B}=[-1,1]\cup\bigcup_{j=1}^k\Gamma_{x_j}.
\end{equation}
At any fixed positive ratio, the distinct off-real transport points
never coincide. Their imaginary parts strictly decrease in $t$, and,
for fixed $t>0$, strictly increase with $x$.
\end{lemma}
\begin{proof}
Apply Lemma~\ref{lem:quadratic-transports} to $b_x=x+i$.
Its formulas give $\tau_{b_x}=\tau_x$, $\beta(b_x)=\beta_x$, and
$\theta_{b_x}=\theta_x$ after direct simplification. For
$t<\tau_x$ the selected inverse is $b_x$; for $t>\tau_x$ it is
not. At $t=\tau_x$ the landing point belongs to the real fixed-ratio
support: otherwise its selected inverse would be real, while
persistence from smaller ratios would give the nonreal value $b_x$.
This proves \eqref{eq:hole-limit-fibers} and
\eqref{eq:hole-limit-graph}. Finally,
\[
 \Imm\theta_x(t)=1-t\left(1+\frac1{x^2+1}\right)
\]
strictly decreases in $t$ and strictly increases in $x$ for fixed
$t>0$. Thus distinct off-real transport points never coincide at
the same positive ratio.
\end{proof}

\subsection{Tracking the small-mass paths}

For the locations in Lemma~\ref{lem:quadratic-spikes}, define
\begin{equation}\label{eq:hole-polynomial-family}
 H_{N,\mathbf x}(z)=(z^2-1)^N
                    \prod_{j=1}^k(z-x_j-i),\qquad
 \eta_N=\frac1{2N+k}\divz H_{N,\mathbf x}.
\end{equation}
For fixed $N$, the shadows of $H_{N,\mathbf x}$ involve all powers
$H_{N,\mathbf x}^n$ of this fixed polynomial; $N$ is not the power level.
As $N\to\infty$, its root law tends to $\eta_0$ while its root-location
set remains $B$. Corollary~\ref{cor:cluster-entire-shadows} gives
\begin{equation}\label{eq:hole-pooled-convergence}
 \Sh_{H_{N,\mathbf x}}\longrightarrow K_{\eta_0,B},\qquad
 \AllSh_{H_{N,\mathbf x}}\longrightarrow K_{\eta_0,B}.
\end{equation}
Also, on every compact interval $I\subset(0,1)$,
\begin{equation}\label{eq:hole-uniform-fibers}
 \sup_{t\in I}d_{\mathrm H}
 \left(S_{t,\eta_N},S_{t,\eta_0}\cup O_t(\eta_0,B)\right)\longrightarrow0.
\end{equation}
Indeed, a contrary sequence has a convergent subsequence of ratios,
and the joint statement of
Corollary~\ref{cor:mass-location-continuity} applies to it.

\begin{lemma}[Paths that follow one transported cluster]
\label{lem:spike-tracking}
Fix $t_0>0$ smaller than every $\tau_{x_j}$. For each $j$ and all large
$N$, there is a continuous definable section
$\gamma_{j,N}(t)\in S_{t,\eta_N}$, $t_0\le t\le1$, starting at a point
converging to $\theta_{x_j}(t_0)$ and ending at the centroid of $\eta_N$.
Any such choice has the following properties:
\begin{enumerate}[label=\textup{(\roman*)},leftmargin=2.2em]
\item On every compact subinterval of $[t_0,\tau_{x_j})$, it converges
uniformly to $\theta_{x_j}(t)$.
\item If $u_N\to u\in[t_0,\tau_{x_j})$, then
\[
 \sup_{\max\{t_0,u_N\}\le t\le1}
 \dist\left(\gamma_{j,N}(t),
       [\theta_{x_j}(u),\beta_{x_j}]\cup[-1,1]\right)
       \longrightarrow0.
\]
\end{enumerate}
\end{lemma}
\begin{proof}
The lower Hausdorff estimate in \eqref{eq:hole-uniform-fibers} supplies
the starting point. Lemmas~\ref{lem:interval-selection}
and~\ref{lem:nested-support-hulls} supply a section ending at the
centroid. Choose $t_1\in(\max_j\tau_{x_j},1)$. Uniform convergence in
\eqref{eq:hole-uniform-fibers} holds on $[t_0,t_1]$, and $S_{t_1,\eta_N}$
approaches a subset of $[-1,1]$. By nested convex hulls, all fibers
with $t\ge t_1$ lie in a horizontal strip of height tending to zero
and with horizontal coordinates in $[-1-o(1),1+o(1)]$.

Here is the separation argument for the earlier ratios. Fix $h>0$.
For an index $\ell$, the portion of $\theta_{x_\ell}$ with imaginary
part at least $h$ and with $t\ge t_0$ is the graph of a function of $t$
on a compact interval. On its time
interval it is separated by a positive distance, uniformly in $t$,
from the real support and every other transport point. This follows
from the strict ordering of imaginary parts in
Lemma~\ref{lem:quadratic-spikes}. For sufficiently large $N$, a section
can therefore neither enter nor leave a small moving neighborhood of
that graph during its time interval: all the support points lie either
in that neighborhood or at a fixed positive distance from it. The set
of times at which the section belongs to the neighborhood is both
open and closed in the interval, so whether the section lies in the
neighborhood is decided at time $t_0$. A graph whose height is at least
$h$ at some later time also has that height at $t_0$, because each
height decreases.

Consequently the $j$th section follows the $j$th graph as long as that
graph stays above height $h$, and it cannot enter any other graph above
height $h$. Every remaining possible limiting point is on $[-1,1]$ or
on a transport segment at height at most $h$. Such a segment point is
within $C_{\mathbf x}h$ of its real landing point, by
\eqref{eq:hole-spike-data}. Uniform convergence excludes negative
imaginary parts bounded away from zero as well. This proves (i), and
proves (ii) first up to an error bounded by $C_{\mathbf x}h$, also
allowing the moving lower endpoint $u_N$. Letting $h\downarrow0$
proves (ii). The convex-hull estimate handles $t\ge t_1$ throughout.
\end{proof}

\subsection{A four-location counterexample}
We first show that the limiting core of a fixed polynomial need not be
contractible.

\begin{theorem}[A robust four-location hole]
\label{thm:noncontractible}\label{conj:contractible}
For every sufficiently large integer $N$, let
\begin{equation}\label{eq:explicit-hole-polynomial}
 P_N(z)=(z^2-1)^N(z-2-i)(z-3-i).
\end{equation}
The point
\[
 q_\diamond=\frac23+\frac{i}{30}
\]
lies in a bounded connected component of both
$\C\setminus\Sh_{P_N}$ and $\C\setminus\AllSh_{P_N}$.
In particular, $\Sh_{P_N}$ is not contractible. These polynomials have
four distinct root locations and no reflection axis for $N\ge2$.
The existence of a hole holds on a nonempty open coefficient
neighborhood of each sufficiently large $P_N$, and hence also for
square-free polynomials with no reflection axis.
\end{theorem}
\begin{proof}
For $x_1=2$ and $x_2=3$, the two segments have real landing points
$\beta_2=2/3$ and $\beta_3=6/11$. They cross transversely at
\begin{equation}\label{eq:explicit-hole-crossing}
 v_* =\frac{30}{37}+\frac4{37}i,
 \qquad
 \theta_2(55/74)=\theta_3(30/37)=v_*.
\end{equation}
Both ratios precede the corresponding landing ratios $5/6$ and
$10/11$. The point $q_\diamond$ lies in the open triangle $\mathcal T_*$ with vertices
$6/11$, $2/3$, and $v_*$ (Figure~\ref{fig:hole-mechanism}). It is outside
$K=[-1,1]\cup\Gamma_2\cup\Gamma_3$, so
\eqref{eq:hole-pooled-convergence} excludes it from \emph{both}
shadows for all large $N$.

\smallskip\noindent\emph{Intersecting paths in the actual core.}
Take the two sections of Lemma~\ref{lem:spike-tracking}, starting at
$t_0=1/2$; thus $\gamma_{1,N}$ follows $\theta_2$ and $\gamma_{2,N}$
follows $\theta_3$. Near the two times in \eqref{eq:explicit-hole-crossing},
the continuous map
\[
 (t,t')\longmapsto\gamma_{1,N}(t)-\gamma_{2,N}(t')
\]
converges uniformly to $\theta_2(t)-\theta_3(t')$. The latter is an
invertible real-affine map, with a zero in the interior of a small
rectangle of time parameters. It follows that the former has a zero
$(t_N,t'_N)$ in that rectangle, with $(t_N,t'_N)\to(55/74,30/37)$.
For completeness, on the rectangle boundary the limiting map has
winding number $1$ or $-1$ about zero. A sufficiently small uniform
perturbation has the same winding number. A continuous map avoiding
zero on the whole rectangle would extend that boundary loop into
$\C\setminus\{0\}$ and force its winding number to be zero. This
proves the asserted intersection without smoothness of the sections.
Let $v_N=\gamma_{1,N}(t_N)=\gamma_{2,N}(t'_N)$; then $v_N\to v_*$.

\smallskip\noindent\emph{Closing the paths and computing their winding.}
Follow the first section from $v_N$ to the common terminal centroid,
then the second section backwards to $v_N$. This gives a continuous
closed curve $\mathcal L_N$ in $\Sh_{P_N}$. Its winding number about $q_\diamond$ is
nonzero. To justify that last assertion, put
\[
 A_1=[v_*,2/3]\cup[-1,1],\qquad
 A_2=[v_*,6/11]\cup[-1,1].
\]
The left horizontal ray from $q_\diamond$ misses $A_1$, and the right
horizontal ray misses $A_2$ (Figure~\ref{fig:hole-mechanism}(b)). By
compactness, fixed sufficiently thin neighborhoods of the two trees avoid their respective rays. By
Lemma~\ref{lem:spike-tracking}(ii), the two parts of $\mathcal L_N$ lie in
these respective neighborhoods for all large $N$.
On the corresponding slit domains, the continuous logarithm branches
\[
 \ell_1(z)=\log(z-q_\diamond),\qquad
 \ell_2(z)=\log(q_\diamond-z)+\pi i
\]
are both logarithms of $z-q_\diamond$.
The common starting point tends to $v_*$, and the terminal centroid
$c_N$ tends to zero. Hence the starting point is above $q_\diamond$
and $c_N$ is below it for all large $N$. At the starting point the
two branches agree, whereas
$\ell_2(c_N)=\ell_1(c_N)+2\pi i$.
Following the first section forwards and the second backwards therefore
gives total logarithmic change $-2\pi i$. Thus $\mathcal L_N$ has winding number
$-1$ about $q_\diamond$, as does $\partial\mathcal T_*$ traversed from $v_*$
to $2/3$ to $6/11$.

Winding number is constant on each component of a loop's complement
and is zero in its unbounded component. Since $q_\diamond$ is outside both
shadows and $\mathcal L_N$ is contained in both, $q_\diamond$ lies in a bounded
complementary component of each. A contraction of $\Sh_{P_N}$ would
contract $\mathcal L_N$ without passing through $q_\diamond$, contrary to its nonzero
winding number.

For $N\ge2$, a reflection preserving the root multiset must preserve
$\{-1,1\}$, the two roots of multiplicity $N$. Its axis would be either
the real or the imaginary axis; neither preserves the other two roots.

\smallskip\noindent\emph{Coefficient stability.}
First fix the two tree neighborhoods avoiding their respective rays
from $q_\diamond$ and a small crossing rectangle with nonzero boundary winding. In the
proof of Lemma~\ref{lem:spike-tracking}, choose a fixed height $h>0$
so small that the low-height pieces lie in these neighborhoods.
Only strict uniform separation and closeness inequalities on the
compact ratio interval $[t_0,t_1]$, with $t_1\in(10/11,1)$ as in that
proof, are then needed to confine each
section's suffix to its tree neighborhood and its crossing part to
its prescribed rectangle. Choose $N$ large enough that all these
inequalities, and exclusion of $q_\diamond$ from the full shadow, hold with
positive margins. Joint fixed-ratio Hausdorff continuity on $[t_0,t_1]$ follows from
Corollary~\ref{cor:mass-location-continuity}, uniformly on this compact
ratio interval, and full-shadow continuity is part of
Proposition~\ref{prop:shadow-foundations}. Hence these margins persist under
all sufficiently small fixed-degree coefficient perturbations. Every
perturbed support family still
has continuous sections by Lemma~\ref{lem:interval-selection}, and
nested convex hulls control all ratios after $t_1$. The same intersection
and winding argument therefore gives a hole throughout an open
neighborhood. The square-free locus is open and dense. The no-reflection locus is
open at $P_N$: otherwise a convergent sequence of root multisets with
reflection axes would give an axis for $P_N$. Indeed, every such axis
passes through the centroid, and its direction has a convergent
subsequence. Intersecting these two loci with the neighborhood gives
the stated perturbations.
\end{proof}

\begin{figure}[htbp]
\centering
\begin{minipage}[b]{0.4688\textwidth}\centering
 \includegraphics[width=\linewidth]{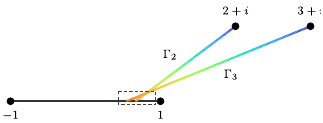}\\[1pt]
 {\small (a)}
\end{minipage}\hfill
\begin{minipage}[b]{0.4995\textwidth}\centering
 \includegraphics[width=\linewidth]{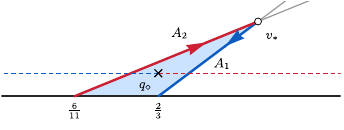}\\[1pt]
 {\small (b)}
\end{minipage}
\caption{(a) The auxiliary limit $K$ for \eqref{eq:explicit-hole-polynomial},
not the core of any fixed $P_N$. Each point of $\Gamma_x$ is colored by the
ratio $t$ at which $\theta_x$ reaches it, as in Figure~\ref{fig:trails}:
the crossing $v_*$ is reached at the different ratios $\frac{55}{74}$
and $\frac{30}{37}$. (b) The dashed box in (a): the open triangle $\mathcal T_*$ (shaded) and
the trees $A_1$, $A_2$ of the proof of Theorem~\ref{thm:noncontractible}; the left ray from $q_\diamond$ misses $A_1$,
the right ray misses $A_2$, and the arrows give the orientation of
$\mathcal L_N$. The proof constructs intersecting support paths for each
sufficiently large fixed $N$ and excludes $q_\diamond$ from the full
shadow; the hole does not follow from this picture or from Hausdorff
convergence alone.}
\label{fig:hole-mechanism}
\end{figure}

\begin{figure}[htbp]
\centering
\includegraphics[width=\textwidth]{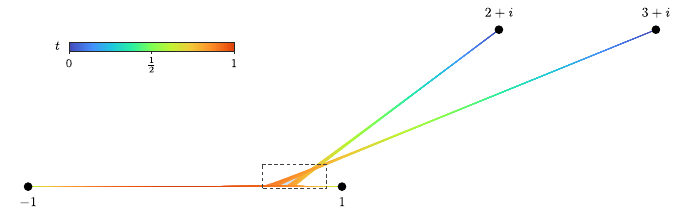}\\[1pt]
{\small (a)}\\[6pt]
\includegraphics[width=\textwidth]{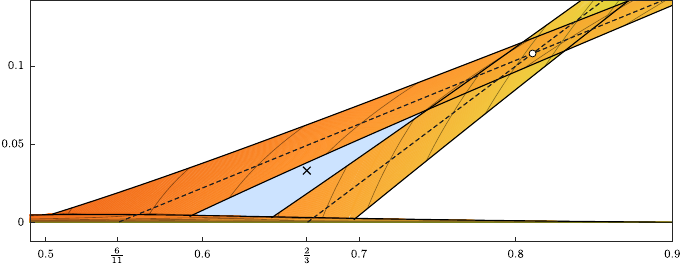}\\[1pt]
{\small (b)}
\caption{The core of $P_N(z)=(z^2-1)^N(z-2-i)(z-3-i)$ for $N=1000$
(degree $d=2002$), as the union of the fixed-ratio supports
$S_{t,\rootmu_{P_N}}$, colored by $t$. (a) The whole core, with
$t$ in steps of $\frac1{400}$: a thin band along $[-1,1]$ and two narrow strips
joining the simple roots $2+i$ and $3+i$ to the real axis. (b) The
dashed box in (a), with $\frac35\le t\le\frac{19}{20}$ in steps of
$\frac1{2000}$ (filled in between); darker arcs: the supports with $t$ a
multiple of $\frac1{40}$. The strips cross and land
on the band, enclosing a hole of the core (shaded) that contains
$q_\diamond$ (cross), hence a hole of the full shadow. Black curves:
the paths of the endpoints of the supports, which are critical values
of $F_{t,\rootmu_{P_N}}$. Dashed: the limit segments $[2/3,2+i]$ and
$[6/11,3+i]$, crossing at $v_*$ (circle). The hole has area about
$0.0019$; as $N\to\infty$ this area tends to $8/1221\approx0.0066$,
while the area of the core tends to zero
(Corollary~\ref{cor:thin-shadow-hole}).}
\label{fig:hole}
\end{figure}

\begin{corollary}[Vanishing shadow area around a persistent planar hole]
\label{cor:thin-shadow-hole}
For the four-location family \eqref{eq:explicit-hole-polynomial},
\[
 \Area(\Sh_{P_N})=\Area(\AllSh_{P_N})\longrightarrow0.
\]
Let $\Omega_N$ be the complementary component containing
$q_\diamond=2/3+i/30$, for either shadow. Then
\[
 \Area(\Omega_N)\longrightarrow\frac8{1221}>0.
\]
Thus holes of positive limiting area can persist while the area of the
shadow itself tends to zero.
\end{corollary}
\begin{proof}
By \eqref{eq:hole-pooled-convergence}, the cores converge in Hausdorff
distance to the compact attachment graph $K$, a finite union of
segments. For each $\varepsilon>0$, all sufficiently late cores lie in
$K_\varepsilon=\{z:\dist(z,K)\le\varepsilon\}$. These neighborhoods
decrease to $K$ as $\varepsilon\downarrow0$; they are bounded and
$\Area(K)=0$. Continuity from above of planar measure therefore gives
$\Area(K_\varepsilon)\to0$, and hence $\Area(\Sh_{P_N})\to0$.
The full shadow has the same area by
\eqref{eq:shadow-interior-area}. Recall the open triangle $\mathcal T_*$
with vertices $6/11$, $2/3$, and $v_*$ from the proof of
Theorem~\ref{thm:noncontractible}. Every point of
$\mathcal T_*$ can be joined to $q_\diamond$ by a compact path in
$\mathcal T_*$. Since $K$ misses that path,
\eqref{eq:hole-pooled-convergence} excludes both shadows from it for
all large $N$. Hence each point of $\mathcal T_*$ eventually belongs to
$\Omega_N$. For a point outside $K\cup\mathcal T_*$, the loop
$\mathcal L_N$ in the proof of Theorem~\ref{thm:noncontractible} has
winding number zero for all large $N$. Indeed, its two parts lie in
arbitrarily thin neighborhoods of the two trees, hence in the
$\varepsilon$-neighborhood of $X=\overline{\mathcal T_*}\cup[-1,1]$ for
every $\varepsilon>0$ and all large $N$. The set $X$ is star convex with
respect to $2/3$, since $\overline{\mathcal T_*}$ and $[-1,1]$ are convex
and contain $2/3$, and therefore so is each of its
$\varepsilon$-neighborhoods. The point has positive distance from $X$,
since $X\subseteq K\cup\mathcal T_*$. For small $\varepsilon$, the
straight-line homotopy from $\mathcal L_N$ to the constant loop $2/3$
stays in the $\varepsilon$-neighborhood of $X$ and never meets the
point. The winding number about $\mathcal L_N$ is constant on the
connected set $\Omega_N$, which misses $\mathcal L_N$, and is nonzero at
$q_\diamond$; so the point cannot belong to $\Omega_N$. Thus $\ind_{\Omega_N}\to\ind_{\mathcal T_*}$ almost
everywhere. All bounded complementary components lie in the fixed
convex hull of $B$, so dominated convergence applies. Finally
$\Area(\mathcal T_*)=\tfrac12(2/3-6/11)(4/37)=8/1221$.
\end{proof}

The finite-$N$ geometry of the theorem and corollary is shown in
Figure~\ref{fig:hole}.

\subsection{Arbitrarily many holes}

\begin{theorem}[Unbounded hole complexity]
\label{thm:many-holes}
For each $k\ge2$, put
\begin{equation}\label{eq:many-hole-polynomials}
 H_{N,k}(z)=(z^2-1)^N\prod_{j=1}^k(z-j-2-i).
\end{equation}
For every sufficiently large $N$, both $\Sh_{H_{N,k}}$ and
$\AllSh_{H_{N,k}}$ have at least $\binom{k}{2}$ distinct holes.
These examples have $k+2$ distinct root locations and degree $2N+k$.
The same lower bound holds on nonempty open sets of square-free
polynomials with no reflection axis, in the corresponding degrees.
In particular, there is no universal bound on the number of holes.
\end{theorem}

\begin{figure}[htbp]
\centering
\includegraphics[width=\textwidth]{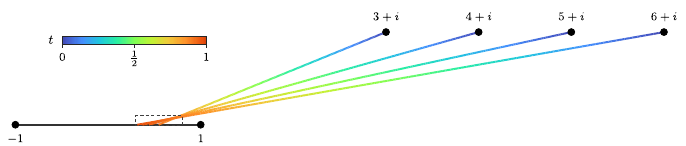}\\[1pt]
{\small (a)}\\[5pt]
\includegraphics[width=\textwidth]{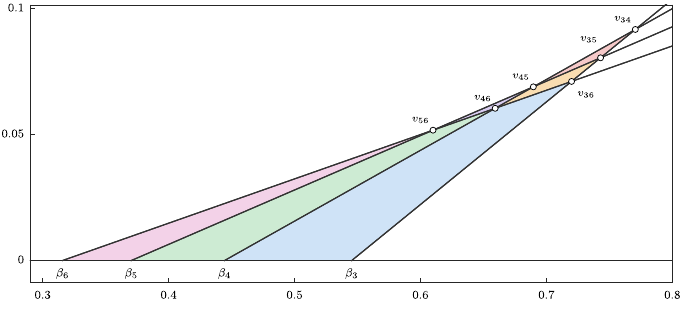}\\[1pt]
{\small (b)}
\caption{The limit graph $G$ in the proof of Theorem~\ref{thm:many-holes}
for $k=4$, that is, for
$H_{N,4}(z)=(z^2-1)^N\prod_{x=3}^{6}(z-x-i)$. (a) Each point of
$\Gamma_x=[\beta_x,x+i]$ is colored by the ratio $t$ at which
$\theta_x$ reaches it. (b) The dashed box in (a), with the vertical
scale doubled: the six crossings $v_{xy}$ of
\eqref{eq:all-hole-crossings} and the $\binom42=6$ bounded faces of $G$
(shaded). Each face lies directly below one crossing, its highest
point; the face points $q_p$ of the proof lie in these faces, near
their apexes.}
\label{fig:many-holes}
\end{figure}

\begin{proof}
Use the preceding construction with $x_j=j+2$. The real landing
points $\beta_{x_j}$ strictly decrease with $j$, whereas the root
abscissae increase. Every two segments therefore cross once,
transversely, at a height strictly between zero and one. Explicitly,
for two abscissae $x<y$, put
\[
 \mathfrak D_{xy}=x^2y^2+2(x^2+xy+y^2).
\]
Their crossing is
\begin{equation}\label{eq:all-hole-crossings}
 v_{xy}=\frac{2xy(x+y)}{\mathfrak D_{xy}}
              +i\frac{2(xy-2)}{\mathfrak D_{xy}}.
\end{equation}
It is reached at the distinct ratios
$\tau_x(1-\Imm v_{xy})$ and $\tau_y(1-\Imm v_{xy})$.
No three of the lines containing the segments $\Gamma_{x_j}$ concur there. Indeed, writing an
intersection as $X+iY$, the equation of the line for abscissa $x$ is
\[
 Yx^3-Xx^2+2x-2X=0.
\]
Three distinct such abscissae, all at least $3$, would be the three
roots of this cubic in $x$. By the sum and product formulas for the
roots, their product would be twice their sum, so the reciprocals of
their three pairwise products would add up to $\frac12$. But each of
these reciprocals is at most $\frac19$, a contradiction. No crossing lies on the base interval or at an
endpoint. Put $G=[-1,1]\cup\bigcup_j\Gamma_{x_j}$. The pair crossings
are distinct, since equality of two distinct pair crossings would give
three concurrent such lines. Figure~\ref{fig:many-holes} shows
$G$ for $k=4$. Index the pairs by $p=(j,j')$,
$j<j'$, and write $v_p=v_{x_jx_{j'}}$ and $h_p=\Imm v_p$. Let
$\mathcal T_p$ be the open triangle with vertices $v_p$,
$\beta_{x_j}$, and $\beta_{x_{j'}}$, and let the closed curve
$\mathcal L_p$ follow segment $j$ downward, the base to segment $j'$,
and segment $j'$ upward. Thus
$\wind(\mathcal L_p,q)=-1$ for $q\in\mathcal T_p$.

For each pair $p$, choose $q_p\in\mathcal T_p$ sufficiently near $v_p$
to avoid every other closed triangle $\overline{\mathcal T_r}$ with
$h_r\le h_p$, and also to avoid $G$. Such a choice is possible:
each closed triangle has only its apex at maximal height, so none
of these other triangles contains $v_p$. Their finite union
is closed, every neighborhood of the apex meets $\mathcal T_p$, and the
finite union of segments $G$ has empty interior.
The square winding-number matrix
\[
 \bigl(\wind(\mathcal L_r,q_p)\bigr)_{r,p}
\]
has diagonal entries $-1$, and its $(r,p)$ entry is zero whenever
$r\ne p$ and $h_r\le h_p$. Ordering pairs by nondecreasing apex
height therefore makes the matrix triangular with nonzero diagonal.
Equal heights cause no difficulty, since all off-diagonal entries
within an equal-height block vanish. The matrix is nonsingular.

Choose a starting ratio $t_0>0$ smaller than every ratio at which a
segment reaches any of its crossings. Lemma~\ref{lem:spike-tracking}
gives $k$ continuous sections. For each pair, the rectangle-winding
argument in the proof of Theorem~\ref{thm:noncontractible} forces their actual
intersection near $v_p$. Concatenate the two suffixes to their
common terminal centroid to form a loop $\mathcal L_{p,N}$ in the core.
All these constructions are simultaneous because there are finitely
many crossings. For each $r$, the filled triangle $\overline{\mathcal T_r}$ together with
the real base interval is compact and star convex. If $r\ne p$ and
$h_r\le h_p$, it avoids $q_p$. Its sufficiently small thickenings
remain star convex and avoid $q_p$; Lemma~\ref{lem:spike-tracking}(ii)
puts the actual loop $\mathcal L_{r,N}$ inside them, so its winding about
$q_p$ is zero.
For the diagonal, the two limiting lower sides lie in opposite
horizontal slit domains about $q_p$. Tracking preserves these slit
conditions, while the crossing stays above $q_p$ and the common
centroid stays below it. The two logarithm branches give winding
minus one. Consequently, for all large $N$ simultaneously,
\[
 \wind(\mathcal L_{p,N},q_p)=-1,\qquad
 \wind(\mathcal L_{r,N},q_p)=0\quad(r\ne p,\ h_r\le h_p).
\]
The actual winding matrix is therefore nonsingular as well.
The points $q_p$ are outside \emph{both} shadows by
\eqref{eq:hole-pooled-convergence}. Each column of the actual winding matrix
is nonzero, so each point is in a bounded complementary component.
Two points in the same complementary component would have identical
windings about every loop in the shadow, giving identical columns.
Nonsingularity excludes this. Thus the indicated bounded components
are distinct for both shadows.

As in the proof of Theorem~\ref{thm:noncontractible}, only finitely many strict
separation and winding conditions are used. They persist on an open
fixed-degree coefficient neighborhood. For $N\ge2$ the only possible
axes preserving the two roots of multiplicity $N$ are the real and
imaginary axes, and neither preserves the remaining upper-half-plane
roots. Passing to square-free polynomials in the open no-axis locus
proves the final assertion.
\end{proof}

Theorems~\ref{thm:noncontractible} and~\ref{thm:many-holes}, together
with Corollary~\ref{cor:thin-shadow-hole}, prove
Theorem~\ref{thm:intro-holes}. The lower bounds hold for every
sufficiently large $N$, although the argument gives no explicit
cutoff. They are consistent with the root-number bound in
Theorem~\ref{thm:atomic-complexity}, since the number of distinct
locations grows with $k$.

These examples also distinguish the individual support paths of
Corollary~\ref{cor:ratio-paths} from a contraction of the core. Every
point has a path to the centroid following increasing ratios, but
those paths cannot always be chosen continuously in their initial
point. The new holes contain no input root and arise from crossings
at different ratios. Finally, transient zeros change neither planar area nor the number
of holes by Proposition~\ref{prop:shadow-foundations}. The constructions above
nevertheless exclude every chosen face point from the full shadow
directly.
\section{Small positive root masses and support strips}
\label{sec:small-atoms}

The hole constructions used only complete zero-set convergence and
continuous paths through the support family. A small \emph{positive}
root mass gives additional local information. It creates an arc of a
fixed-ratio support near the selected transport point. We determine
its scale and density, and give a condition under which varying the
ratio sweeps out planar interior of the core.

For complex vectors $u,v$, write
$\det_{\R}(u,v)=\Imm(\overline u v)$ for their real determinant.
A \emph{support strip} is an open planar set obtained as the image of
a two-parameter family of fixed-ratio support arcs under a
parametrization whose real differential has rank two. The two parameters will be the derivative
ratio and a coordinate along the arc.

\subsection{Local support near a transported small atom}

Let $\eta$ be a probability measure on a finite set $A$, and let
$b_1,\ldots,b_k\notin A$ be distinct. Fix $\omega_j>0$, put
$\omega_{\mathrm{tot}}=\sum_j \omega_j$, and set
\begin{equation}\label{eq:small-atom-law}
 \eta_\varepsilon=(1-\varepsilon \omega_{\mathrm{tot}})\eta
                  +\varepsilon\sum_{j=1}^k \omega_j\delta_{b_j},
 \qquad 0<\varepsilon<\omega_{\mathrm{tot}}^{-1}.
\end{equation}
Thus every added location has positive mass, however small. All the
limiting supports below are supports for this positive-mass law, not
zero sets of individual polynomials of an array.

\begin{theorem}[Persistence of a transported small atom]
\label{thm:small-atom-strip}
Fix $b=b_j$ with $\chi_b:=C_\eta(b)\ne0$. Let $J$ be a compact interval
in $(0,1)$ on a neighborhood of which
\[
 \theta_b(t):=b-\frac{t}{\chi_b}\notin S_{t,\eta},
 \qquad W_{t,\eta}(\theta_b(t))=b.
\]
Assume also that $\theta_b(J)$ avoids $A\cup\{b_1,\ldots,b_k\}$.
Put
\begin{equation}\label{eq:small-atom-scales}
 \mathfrak d_b(t)=1+\frac{tC_\eta'(b)}{\chi_b^2},\qquad
 \mathfrak a_b(t)=\frac1{\chi_b}\sqrt{\frac{t}{\mathfrak d_b(t)}},\qquad
 \lambda_b(t)=\mathfrak d_b(t)\mathfrak a_b(t).
\end{equation}
The selected inverse is simple, so $\mathfrak d_b(t)\ne0$; choose a
real-analytic square root on a neighborhood of $J$.
There are $0<r<2$, $\varepsilon_0>0$, and real-valued functions
$h_\varepsilon(t,x)$, real-analytic in $(t,x,\sqrt\varepsilon)$ through
$\varepsilon=0$, with $h_0=0$, such that
\begin{equation}\label{eq:small-atom-strip-map}
 X_\varepsilon(t,x)=\theta_b(t)+\sqrt{\varepsilon \omega_j}\,
           \lambda_b(t)\bigl(x+i h_\varepsilon(t,x)\bigr)
       \in S_{t,\eta_\varepsilon}
 \quad(t\in J,\ |x|<r,
               \ 0<\varepsilon<\varepsilon_0).
\end{equation}
The center curve $\gamma_{b,\varepsilon}(t):=X_\varepsilon(t,0)$ satisfies
\begin{equation}\label{eq:small-atom-center}
 \|\gamma_{b,\varepsilon}-\theta_b\|_{C^1(J)}=O(\varepsilon).
\end{equation}
If
\begin{equation}\label{eq:small-atom-transversality}
 \det_{\R}\bigl(\theta_b'(t),\lambda_b(t)\bigr)\ne0
                  \quad(t\in J),
\end{equation}
then these center curves lie in $\operatorname{int}S_{\eta_\varepsilon}$
for all sufficiently small $\varepsilon>0$. More precisely, the map in
\eqref{eq:small-atom-strip-map} has rank two near $J\times\{0\}$.
These interior neighborhoods persist under small coefficient
perturbations whenever $\eta_\varepsilon$ is a polynomial root law.
No simplicity condition on the input roots or on unselected saddles
is required.
\end{theorem}
\begin{proof}
\emph{The local phase.}
All constructions take place on a slightly larger compact interval
than $J$. They can therefore be made uniform by a finite cover in
$t$. Write $\omega=\omega_j$, $\chi=\chi_b$, $\mathfrak d=\mathfrak d_b(t)$, $\mathfrak a=\mathfrak a_b(t)$,
$\lambda=\lambda_b(t)$, and $\sigma=\sqrt{\varepsilon \omega}$.
Choose local analytic logarithms and put
\[
 V(w)=\sum_{a_0\in A}\eta(\{a_0\})\log(w-a_0),\qquad
 V_\varepsilon^{\mathrm{reg}}(w)
   =(1-\varepsilon \omega_{\mathrm{tot}})V(w)
                   +\varepsilon\sum_{\ell\ne j}\omega_\ell\log(w-b_\ell).
\]
The second function is holomorphic near $b$. The perturbed phase is
\[
 \Phi_\varepsilon(w,z,t)=V_\varepsilon^{\mathrm{reg}}(w)
                          +\varepsilon \omega\log(w-b)-t\log(w-z).
\]
At $\varepsilon=0$ and $z=\theta_b(t)$ the selected saddle is $b$,
and
\begin{equation}\label{eq:small-atom-Hessian}
 \Phi_{0,ww}(b,\theta_b(t),t)=\frac{\chi^2\mathfrak d}{t}=\mathfrak a^{-2},
 \qquad
 \Phi_{0,wz}(b,\theta_b(t),t)=-\frac{\chi^2}{t}.
\end{equation}
Make the substitutions $w=b+\sigma \mathfrak a y$, $z=\theta_b(t)+\sigma\lambda\xi$.
After subtracting the common term
$V_\varepsilon^{\mathrm{reg}}(b)-t\log(b-z)+\varepsilon \omega\log(\sigma\mathfrak a)$
and dividing by $\varepsilon \omega=\sigma^2$, Taylor's formula gives
\begin{equation}\label{eq:small-atom-model}
 \Psi_\sigma(y,\xi,t)=\log y+\frac{y^2}{2}-\xi y+O(\sigma).
\end{equation}
The remainder extends analytically through $\sigma=0$, with all local
derivatives, on compact subsets with $y\ne0$. In particular this
statement is uniform near $(y,\xi)=(i,0)$ and $(-i,0)$.
The model has two simple saddles there. They continue as
$y_+(\xi,t,\sigma)$ and $y_-(\xi,t,\sigma)$, with $y_\pm(0,t,0)=\pm i$.

\smallskip\noindent\emph{Selection of the two-saddle maximum.}
We verify that both local saddles are selected at their common
height; an expansion of the saddle equation alone would not prove
\eqref{eq:small-atom-strip-map}. At $\varepsilon=0$, the observation
point $\theta_b(t)$ has a one-saddle separating contour through $b$ and increasing paths from its two higher sectors
to the two positive ends. Their existence follows from the weighted
harmonic-barrier construction in
Proposition~\ref{prop:weighted-foundations}. Away from $b$, deform
the contour and the compact parts of those paths to avoid every
other added location. Their strict height gaps are preserved.
For fixed small neighborhoods of the positive ends the signs of
the end strengths remain unchanged for small $\varepsilon$.

Here are local replacements for the passage through $b$. At
$\xi=\sigma=0$ the real part of the model in
\eqref{eq:small-atom-model} is
\[
 \mathscr L(y)=\log|y|+\tfrac12\Ree(y^2).
\]
On the imaginary axis $y=i\varrho$ it equals
$\log|\varrho|-\varrho^2/2$, with its only maxima at $\varrho=\pm1$, both of
height $-1/2$. Join the portions of this axis on either side of
zero by a semicircle of radius $1/4$. On that semicircle
$\mathscr L\le\log(1/4)+1/32<-1/2$. This is a contour passage through
both saddles, below their height everywhere else. At either saddle,
the horizontal path $y=x\pm i$ satisfies
\[
 \mathscr L(x\pm i)=\frac{x^2-1}{2}+\frac12\log(1+x^2),
\]
which increases strictly with $|x|$. Its two ends lead to the
opposite increasing sectors of the base saddle.

For completeness, these replacements can be matched to the fixed
base contour and paths despite their different spatial scales.
Use a Morse coordinate $v$ with
$\Phi_0(w,\theta_b(t),t)-\Phi_0(b,\theta_b(t),t)=v^2/2$ and
$v'(b)=\mathfrak a^{-1}$. In the local replacements $v/\sigma=y+O(\sigma)$.
Choose a large fixed $R$, then a small fixed Morse radius $\delta$.
Between $R\le |v|/\sigma\le\delta/\sigma$, the derivative of the real
perturbed phase along the decreasing imaginary passages, after
writing $|v|/\sigma=\varrho$, is
\[
 -\sigma^2\varrho+O(\sigma^2/\varrho+\sigma^2|\xi|+\sigma^3).
\]
Along the outgoing passages in either increasing sector it is
$\sigma^2\varrho+O(\sigma^2/\varrho+\sigma^2|\xi|+\sigma^3)$, with an immaterial bounded
imaginary offset in the scaled coordinate. These estimates follow
from the exact quadratic base phase, the logarithmic term, and
uniform bounds for the remaining analytic derivatives. Hence the real
phase is strictly decreasing and increasing along these passages,
respectively, once $R$ is large
and the $\xi$-disk and $\sigma$ are small. At the fixed Morse radius they
join the previously chosen contour and paths, whose height margins
are fixed. On the rest of these compact curves the perturbation
converges uniformly to zero.

On the bounded scaled pieces, the two simple saddles and their
increasing and decreasing directions persist by the Morse lemma.
The modified piecewise smooth loop retains winding one about $z$
and has maximum equal to the larger of the two local saddle heights.
Each saddle also has an increasing path to each positive end, with
minimum equal to its own height. The upper and lower barrier
inequalities therefore give exactly the maximum of these two heights.
The strict height inequality on the rest of the loop excludes every
other selected saddle, while at a common height the increasing paths
select both local saddles. Now apply the weighted barrier-contour
construction of Proposition~\ref{prop:weighted-contours} to the
perturbed law. It gives a positively oriented simple Jordan contour
at this same maximum, through every selected saddle and strictly
lower elsewhere. At a tie, that is, when the two local saddle heights
are equal, it passes through both local saddles,
each exactly once, smoothly and regularly, with a nondegenerate
maximum there. Compactness gives a strict height gap off any fixed
neighborhoods of the selected saddles. The parameter box, including
one fixed small $\xi$-disk, can be chosen uniformly for $t\in J$
by a finite cover.

\smallskip\noindent\emph{The support arc and its projection.}
Let $\psi_\pm(\xi,t,\sigma)=\Psi_\sigma(y_\pm(\xi,t,\sigma),\xi,t)$ and
$g_\sigma=\Ree(\psi_+-\psi_-)$. At $\sigma=0$ the real $\xi$-axis is a
tie curve near zero. Write $\xi=\xi_1+i\xi_2$ and
$\nabla_\xi=(\partial_{\xi_1},\partial_{\xi_2})$ for the real gradient.
Since the derivative of a critical value in $\xi$ is $-y$ in the model,
\[
 \partial_{\xi_2} g_0(0,t)=2.
\]
The real-analytic implicit-function theorem writes the selected tie
as $\xi=x+i h_\varepsilon(t,x)$, with
$h_\varepsilon=O(\sqrt\varepsilon)$ in $C^1$ on smaller rectangles.
The difference of the two critical heights has nonzero spatial
gradient there. The Laplacian of the selected
maximum has positive arclength density there, so the weighted
potential formula places every point of this tie in
$S_{t,\eta_\varepsilon}$. This proves
\eqref{eq:small-atom-strip-map} and \eqref{eq:small-atom-center}.

At $x=0$, differentiation gives
\begin{equation}\label{eq:small-atom-strip-Jacobian}
 \det_{\R}\bigl((X_\varepsilon)_t,(X_\varepsilon)_x\bigr)
   =\sqrt{\varepsilon \omega}\,
       \det_{\R}\bigl(\theta_b'(t),\lambda_b(t)\bigr)+O(\varepsilon).
\end{equation}
Under \eqref{eq:small-atom-transversality}, compactness makes this
nonzero near $J\times\{0\}$ for small $\varepsilon$. The inverse-function
theorem then gives open subsets of the limiting core $S_{\eta_\varepsilon}$ around the center
curve. At fixed $\varepsilon>0$, both the two-simple-saddle selected
maximum and this rank condition persist in the polynomial
coefficients. A finite cover of a compact center-curve portion gives
persistence of its interior neighborhoods, including under splitting
of multiple input roots.
\end{proof}

\begin{corollary}[The local semicircle profile of a small root mass]
\label{cor:small-atom-density}
Let $r$ and $X_\varepsilon$ be as in Theorem~\ref{thm:small-atom-strip}.
For fixed $t$, the map $x\mapsto X_\varepsilon(t,x)$ is one-to-one.
Define a measure on $(-r,r)$ by
\[
 \mathfrak m_{\varepsilon,t}(E)=\mu_{t,\eta_\varepsilon}
       \bigl(\{X_\varepsilon(t,x):x\in E\}\bigr)
       \qquad(E\subset(-r,r)\text{ Borel}).
\]
It is absolutely continuous, and on every smaller compact rectangle
in $(t,x)$,
\begin{equation}\label{eq:small-atom-density}
 \frac{\dd \mathfrak m_{\varepsilon,t}}{\dd x}(x)
   =\frac{\varepsilon \omega_j}{2\pi(1-t)}
          \left(\sqrt{4-x^2}+O(\sqrt\varepsilon)\right).
\end{equation}
Thus, after normalization by $(1-t)/(\varepsilon \omega_j)$, the local
profile converges to the restriction of the \emph{semicircle law},
the probability measure with density
$(2\pi)^{-1}\sqrt{(4-x^2)_+}$ on the real line.
This is a local assertion on $(-r,r)$; no conclusion at the two
endpoints $x=\pm2$ is needed.
\end{corollary}
\begin{proof}
In the notation of the preceding proof the difference of the two
physical critical heights is $\sigma^2g_\sigma$, where $\sigma^2=\varepsilon \omega_j$.
Its spatial gradient has modulus
$\sigma|\lambda_b|^{-1}|\nabla_\xi g_\sigma|$. Arclength along the tie is
$\sigma|\lambda_b|\sqrt{1+(\partial_xh_\varepsilon)^2}\,\dd x$.
The distributional Laplacian of a maximum of two harmonic functions
has density equal to the jump in normal derivative; the potential
normalization divides this by $2\pi(1-t)$. Consequently
\[
 \frac{\dd \mathfrak m_{\varepsilon,t}}{\dd x}
 =\frac{\sigma^2}{2\pi(1-t)}
  |\nabla_\xi g_\sigma(x+i h_\varepsilon,t)|
                   \sqrt{1+(\partial_xh_\varepsilon)^2}.
\]
At $\sigma=0$ the two model saddles over real $x$ are
$y_\pm=(x\pm i\sqrt{4-x^2})/2$. The difference of their critical
value derivatives is $-y_++y_-=-i\sqrt{4-x^2}$.
The analytic estimates in the theorem give
$|\nabla_\xi g_\sigma|=\sqrt{4-x^2}+O(\sigma)$ and
$\partial_xh_\varepsilon=O(\sigma)$, uniformly on compact rectangles.
This proves \eqref{eq:small-atom-density}.
\end{proof}

Figure~\ref{fig:small-mass} illustrates Theorem~\ref{thm:small-atom-strip}
and Corollary~\ref{cor:small-atom-density} for $\eta=\eta_0$, the base
measure of Section~\ref{sec:holes}, and one added location $b=2+i$.

\begin{figure}[htbp]
\centering
\includegraphics[width=0.7229\textwidth]{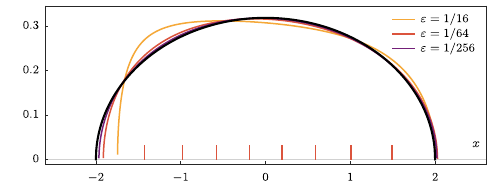}\\[1pt]
{\small (a)}\\[5pt]
\includegraphics[width=0.7229\textwidth]{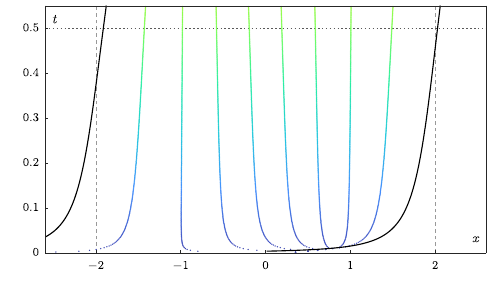}\\[1pt]
{\small (b)}
\caption{A small positive mass at $b=2+i$ over $\eta_0$, in the
coordinates $z=\theta_b(t)+\sqrt\varepsilon\,\lambda_b(t)(x+iy)$ of
Theorem~\ref{thm:small-atom-strip} (with $k=1$, $b_1=b$, $\omega_1=1$). (a) The density of
$\mu_{1/2,\eta_\varepsilon}$ along its arc near $\theta_b(\frac12)$,
normalized as in Corollary~\ref{cor:small-atom-density}, for
$\varepsilon=\frac1{16},\frac1{64},\frac1{256}$, and the semicircle
density $\sqrt{4-x^2}/(2\pi)$ (black). (b) The zeros of all derivatives
$H^{(m)}$ of $H(z)=(z^2-1)^{252}(z-2-i)^8$ with $|x+iy|<3$, at height
$t=m/D$, $D=512$; the root law of $H$ is $\eta_\varepsilon$ with
$\varepsilon=\frac1{64}$. Each order contributes eight zeros. Away from
$t=0$ they keep nearly fixed positions $x$ (ticks in (a): $t=\frac12$,
dotted), between the paths of the endpoints of the arcs of
$S_{t,\eta_\varepsilon}$ (black); dashed: $x=\pm2$.}
\label{fig:small-mass}
\end{figure}

\subsection{Transversality for the quadratic base}

\begin{lemma}[Automatic transversality over two real roots]
\label{lem:quadratic-small-atoms}
For $\eta_0=(\delta_{-1}+\delta_1)/2$ and $b=x+iy$ with $x,y>0$,
the selected path of Lemma~\ref{lem:quadratic-transports} satisfies
\eqref{eq:small-atom-transversality} throughout $0<t<\tau_b$.
Consequently Theorem~\ref{thm:small-atom-strip} applies on every
compact subinterval for which the path avoids the finitely many
added locations; no further selection or transversality test is
required.
\end{lemma}
\begin{proof}
The saddle derivative in \eqref{eq:small-atom-scales} is
\[
 \mathfrak d_b(t)=1-t\left(1+\frac1{b^2}\right).
\]
Its imaginary part is strictly positive for $x,y,t>0$. For either
square-root choice,
\[
 \left(\frac{\lambda_b(t)}{\theta_b'(t)}\right)^2=t\mathfrak d_b(t).
\]
The square has strictly positive imaginary part, so the ratio is nonreal,
which is precisely
\eqref{eq:small-atom-transversality}. The remaining hypotheses were
verified in Lemma~\ref{lem:quadratic-transports}.
\end{proof}

\begin{corollary}[Interior crossings at different ratios]
\label{cor:small-atom-crossing}
Apply Theorem~\ref{thm:small-atom-strip} to two added locations $b_j$
and $b_{j'}$, with compact ratio intervals having interior points $t_j$
and $t_{j'}$. Suppose
\[
 \theta_{b_j}(t_j)=\theta_{b_{j'}}(t_{j'})=z_*,\qquad
 \det_{\R}\bigl(\theta_{b_j}'(t_j),\theta_{b_{j'}}'(t_{j'})\bigr)\ne0,
\]
and suppose the strip-transversality condition
\eqref{eq:small-atom-transversality} holds at both points.
For all sufficiently small $\varepsilon>0$, the two center curves have
an intersection $z_\varepsilon$ in $\operatorname{int}S_{\eta_\varepsilon}$,
at ratios $t_{j,\varepsilon}$ and $t_{j',\varepsilon}$, with
\[
 z_\varepsilon=z_*+O(\varepsilon),\qquad
 t_{j,\varepsilon}=t_j+O(\varepsilon),\qquad
 t_{j',\varepsilon}=t_{j'}+O(\varepsilon).
\]
The intersection is transverse and persists under sufficiently small
coefficient perturbations when $\eta_\varepsilon$ is a polynomial root law.
\end{corollary}
\begin{proof}
Consider the real map
\[
 (t,t')\longmapsto
 \gamma_{b_j,\varepsilon}(t)-\gamma_{b_{j'},\varepsilon}(t').
\]
At $\varepsilon=0$ its differential at $(t_j,t_{j'})$ has columns
$\theta_{b_j}'(t_j)$ and $-\theta_{b_{j'}}'(t_{j'})$ and is invertible.
The center curves depend real-analytically on $\sqrt\varepsilon$ and
are $O(\varepsilon)$-close to their limiting paths in $C^1$, by
\eqref{eq:small-atom-center}. The implicit-function theorem gives
the intersection and the stated estimates. The rank-two conclusion
of Theorem~\ref{thm:small-atom-strip} puts it in the planar interior.
Invertibility and the selected two-saddle charts persist in the
coefficients, proving the last assertion.
\end{proof}

\begin{remark}[Smooth paths in the hole examples]
For the polynomial family \eqref{eq:hole-polynomial-family}, take
$\varepsilon=(2N+k)^{-1}$ and $\omega_j=1$ in
\eqref{eq:small-atom-law}. Its positive-mass law is exactly $\eta_N$.
The two strips in Figure~\ref{fig:hole} are swept out by the
corresponding arcs, for $\varepsilon=1/2002$.
On compact ratio intervals preceding the real landings, the paths of
Lemma~\ref{lem:spike-tracking} can therefore be chosen to be the
center curves of Theorem~\ref{thm:small-atom-strip}, and continued to
the centroid by Lemma~\ref{lem:interval-selection}. Their compact
center-curve portions lie in the planar interior, and
Corollary~\ref{cor:small-atom-crossing} makes the pairwise crossings
transverse interior crossings. In the four-location example, one may
start at $t=1/2$ and use such portions until after the crossing times
in \eqref{eq:explicit-hole-crossing}.

This strengthens the local geometry but is not needed to close the
loops. A local crossing alone does not produce a hole: the continuation
to the common centroid and exclusion of the face points are the
separate steps established in the preceding section.
\end{remark}
\section{Sharp topology bounds}\label{sec:hole-questions}
\begin{problem}[Minimal degree and sharp hole bounds]\label{prob:hole-bounds}
Determine the smallest degree admitting a limiting core with a hole,
and the largest possible number of holes for a prescribed degree or
number of distinct roots. Describe which changes in the support family
and in the selected transport points create or destroy a hole.
\end{problem}
Theorem~\ref{thm:noncontractible} establishes that a hole can occur. Theorem~\ref{thm:many-holes} gives
at least $\binom{s-2}{2}$ holes using $s$ distinct root locations and
sufficiently large multiplicities at two of them. These lower bounds
complement the finite root-number topology bounds of
Theorem~\ref{thm:atomic-complexity}. The challenge is to make the upper bounds effective and to determine
which of the lower bounds are sharp.

\end{document}